\documentclass[12pt,a4paper]{article}

\usepackage{amsmath,amssymb,amsthm,float}
\usepackage[dvips]{graphicx}

\numberwithin{equation}{section}
\numberwithin{figure}{section}

\newtheorem{theorem}{Theorem}[section]
\newtheorem{lemma}[theorem]{Lemma}
\newtheorem{proposition}[theorem]{Proposition}
\newtheorem{remark}[theorem]{Remark}

\begin{document}

\title{\textbf{Graphical translating solitons \\for the mean curvature flow \\and isoparametric functions}}
\author{Tomoki Fujii}
\date{Department of Mathematics, Tokyo University of Science}
\maketitle

\begin{abstract}
In this paper, we consider a translating soliton for the mean curvature flow starting from a graph of a function on a domain in a unit sphere which is constant along each leaf of isoparametric foliation.
First, we show that such a function is given as a composition of an isoparametric function on the sphere and a function which is given as a solution of  a certain ordinary differential equation.
Further, we analyze the shape of the graphs of the solutions of the ordinary differential equation.
This analysis leads to the classification of the shape of such translating solitons.
Finally, we investigate a domain of the function which is given as  a composition of the isoparametric function and the solution of the ordinary differential equation in the case where the number of distinct principal curvatures of the isoparametric hypersurface defined by the regular level set for the isoparametric function is $1$, $2$, or $3$.
\end{abstract}

\section{Introduction}

Let $N$ be a $n$-dimensional Riemannian manifold and $u:M\to\mathbb{R}$ be a function on a domain $M\subset N$.
Define the immersion $f$ of $M$ into the product Riemannian manifold $N\times\mathbb{R}$ by $f(x)=(x,u(x)),~x\in M$.
Denote the graph of $u$ by $\Gamma$.
If a $C^{\infty}$-familly of $C^{\infty}$-immersions $\{f_t\}_{t\in I}$ of $M$ into $N\times\mathbb{R}$ ($I$ is an open interval including $0$) satisfies 
\begin{equation}\label{eq:mcf}
\begin{cases}
\displaystyle\left(\frac{\partial f_t}{\partial t}\right)^{\bot_{f_t}}=H_t\\
f_0=f,
\end{cases}
\end{equation}
as $M_t=f_t(M)$, $\{M_t\}_{t\in I}$ is called the mean curvature flow starting from $\Gamma$.
Here, $H_t$ is the mean curvature vector field of $f_t$ and $(\bullet)^{\bot_{f_t}}$ is the normal component of $(\bullet)$ with respect to $f_t$.
Further, according to Hungerb\"{u}hler and Smoczyk\cite{HS}, we define a soliton of the mean curvature flow.
Let $X$ be a Killing vector field on $N\times\mathbb{R}$ and $\{\phi_t\}_{t\in\mathbb{R}}$ be the one-parameter transformation associated to $X$ on $N\times\mathbb{R}$, that is, $\phi_t$ satisfies
\begin{equation*}
\begin{cases}
\displaystyle\frac{\partial \phi_t}{\partial t}=X\circ\phi_t\\
\phi_0=id_{N\times\mathbb{R}},
\end{cases}
\end{equation*}
where $id_{N\times\mathbb{R}}$ is the identity map on $N\times\mathbb{R}$.
Here, we note that $\phi_t$'s are isometries.
Then, the mean curvature flow $\{M_t\}_{t\in I}$ is called a \textit{soliton of the mean curvature flow with respect to $X$}, if  $\widetilde{f}_t=\phi_t^{-1}\circ f_t$ satisfies
\begin{equation}
\left(\frac{\partial \widetilde{f}_t}{\partial t}\right)^{\bot_{\widetilde{f}_t}}=0.\label{eq:mcf-soliton}
\end{equation}
In the sequel, we call such a soliton a \textit{$X$-soliton} simply.
In particular, when $X=(0,1)\in T(N\times\mathbb{R})=TN\oplus T\mathbb{R}$, we call the $X$-soliton a \textit{translating soliton}.

The translating soliton for $N=\mathbb{R}^n$ has been studied by several authors.
When $n=2$, Shahriyari\cite{S} proved non-existence of complete translating graphs over bounded connected domains of $\mathbb{R}^2$ with smooth boundary.
Also, she showed that if a complete translating soliton which is a graph over a domain in $\mathbb{R}^2$, then the domain is a strip, or a halfspace, or $\mathbb{R}^2$.
Further, Hoffman, Ilmanen, Mart\'{i}n and White\cite{HIMW} showed that no complete translating soliton is the graph of a function over a halfspace in $\mathbb{R}^2$.
For the function $u:\mathbb{R}^n\to\mathbb{R}$ defined by $u(x_1,\cdots,x_n)=-\log{\cos{x_n}}$,~$(x_1,\cdots,x_n)\in\mathbb{R}^n$, the mean curvature flow starting from the graph of $u$ is the translating soliton.
When $n=1$, the curve of $u$ is called a {\it grim reaper} (Figure \ref{grim-reaper}).
When $n\ge2$, the graph of $u$ is called a {\it grim hyperplane} (Figure \ref{grim-hyperplane}).
Further, Martin, Savas-Halilaj, and Smoczyk\cite{MSS} gave the characterization of the grim hyperplane.
Clutterbuck, Schn\'{u}rer and Schulze\cite{CSS} showed the existence of  the complete rotationally symmetric graphical translating soliton which is called bowl soliton (Altschuler and Wu\cite{AW} had already showed the existence in the case $n=2$) and a certain type of stability for the bowl soliton.
Further, they showed that bowl solitons have the following asymptotic expansion as $r$ approaches infinity:
$$
\frac{r^2}{2(n-1)}-\log{r}+O(r^{-1}),
$$
where $r$ is the distance function in $\mathbb{R}^n$ because $u$ is the composition of $r$ and the solution of a certain ordinary differential equation.
Wang\cite{W} showd that the bowl soliton is the only convex translating soliton which is an entire graph.
Further, Spruck and Xiao\cite{SX} showed that the bowl soliton is the only complete translating soliton which is an entire graph.
In this paper, we consider the case where the symmetry of the graph is a little complicated.

\begin{figure}[H]
\centering
\begin{minipage}{0.49\columnwidth}
\centering
\scalebox{0.8}{{\unitlength 0.1in%
\begin{picture}(20.0000,20.0000)(2.0000,-22.0000)%
\special{pn 20}%
\special{pa 708 200}%
\special{pa 710 297}%
\special{pa 715 473}%
\special{pa 720 598}%
\special{pa 725 695}%
\special{pa 730 774}%
\special{pa 735 840}%
\special{pa 740 898}%
\special{pa 745 949}%
\special{pa 750 994}%
\special{pa 755 1035}%
\special{pa 760 1073}%
\special{pa 765 1107}%
\special{pa 770 1139}%
\special{pa 775 1168}%
\special{pa 780 1196}%
\special{pa 790 1246}%
\special{pa 800 1290}%
\special{pa 810 1330}%
\special{pa 820 1366}%
\special{pa 825 1383}%
\special{pa 830 1399}%
\special{pa 840 1429}%
\special{pa 850 1457}%
\special{pa 860 1483}%
\special{pa 870 1507}%
\special{pa 885 1540}%
\special{pa 895 1560}%
\special{pa 915 1596}%
\special{pa 920 1604}%
\special{pa 925 1613}%
\special{pa 930 1620}%
\special{pa 935 1628}%
\special{pa 940 1635}%
\special{pa 945 1643}%
\special{pa 950 1649}%
\special{pa 960 1663}%
\special{pa 980 1687}%
\special{pa 985 1692}%
\special{pa 990 1698}%
\special{pa 1010 1718}%
\special{pa 1015 1722}%
\special{pa 1020 1727}%
\special{pa 1040 1743}%
\special{pa 1045 1746}%
\special{pa 1050 1750}%
\special{pa 1055 1753}%
\special{pa 1060 1757}%
\special{pa 1075 1766}%
\special{pa 1080 1768}%
\special{pa 1090 1774}%
\special{pa 1125 1788}%
\special{pa 1130 1789}%
\special{pa 1135 1791}%
\special{pa 1145 1793}%
\special{pa 1150 1795}%
\special{pa 1160 1797}%
\special{pa 1165 1797}%
\special{pa 1175 1799}%
\special{pa 1180 1799}%
\special{pa 1185 1800}%
\special{pa 1215 1800}%
\special{pa 1220 1799}%
\special{pa 1225 1799}%
\special{pa 1235 1797}%
\special{pa 1240 1797}%
\special{pa 1250 1795}%
\special{pa 1255 1793}%
\special{pa 1265 1791}%
\special{pa 1270 1789}%
\special{pa 1275 1788}%
\special{pa 1310 1774}%
\special{pa 1320 1768}%
\special{pa 1325 1766}%
\special{pa 1340 1757}%
\special{pa 1345 1753}%
\special{pa 1350 1750}%
\special{pa 1355 1746}%
\special{pa 1360 1743}%
\special{pa 1380 1727}%
\special{pa 1385 1722}%
\special{pa 1390 1718}%
\special{pa 1410 1698}%
\special{pa 1415 1692}%
\special{pa 1420 1687}%
\special{pa 1440 1663}%
\special{pa 1450 1649}%
\special{pa 1455 1643}%
\special{pa 1460 1635}%
\special{pa 1465 1628}%
\special{pa 1470 1620}%
\special{pa 1475 1613}%
\special{pa 1480 1604}%
\special{pa 1485 1596}%
\special{pa 1505 1560}%
\special{pa 1515 1540}%
\special{pa 1530 1507}%
\special{pa 1540 1483}%
\special{pa 1550 1457}%
\special{pa 1560 1429}%
\special{pa 1570 1399}%
\special{pa 1575 1383}%
\special{pa 1580 1366}%
\special{pa 1590 1330}%
\special{pa 1600 1290}%
\special{pa 1610 1246}%
\special{pa 1620 1196}%
\special{pa 1625 1168}%
\special{pa 1630 1139}%
\special{pa 1635 1107}%
\special{pa 1640 1073}%
\special{pa 1645 1035}%
\special{pa 1650 994}%
\special{pa 1655 949}%
\special{pa 1660 898}%
\special{pa 1665 840}%
\special{pa 1670 774}%
\special{pa 1675 695}%
\special{pa 1680 598}%
\special{pa 1685 473}%
\special{pa 1690 297}%
\special{pa 1692 200}%
\special{fp}%
%
\special{pn 8}%
\special{pa 400 1800}%
\special{pa 2000 1800}%
\special{fp}%
\special{sh 1}%
\special{pa 2000 1800}%
\special{pa 1933 1780}%
\special{pa 1947 1800}%
\special{pa 1933 1820}%
\special{pa 2000 1800}%
\special{fp}%
\special{pa 1200 2200}%
\special{pa 1200 200}%
\special{fp}%
\special{sh 1}%
\special{pa 1200 200}%
\special{pa 1180 267}%
\special{pa 1200 253}%
\special{pa 1220 267}%
\special{pa 1200 200}%
\special{fp}%
\put(12.6000,-19.6000){\makebox(0,0)[lb]{$O$}}%
\put(20.4000,-18.5000){\makebox(0,0)[lb]{$x$}}%
%
\special{pn 20}%
\special{pa 1200 1600}%
\special{pa 1200 1200}%
\special{fp}%
\special{sh 1}%
\special{pa 1200 1200}%
\special{pa 1180 1267}%
\special{pa 1200 1253}%
\special{pa 1220 1267}%
\special{pa 1200 1200}%
\special{fp}%
\special{pa 1400 1200}%
\special{pa 1400 800}%
\special{fp}%
\special{sh 1}%
\special{pa 1400 800}%
\special{pa 1380 867}%
\special{pa 1400 853}%
\special{pa 1420 867}%
\special{pa 1400 800}%
\special{fp}%
\special{pa 1000 1200}%
\special{pa 1000 800}%
\special{fp}%
\special{sh 1}%
\special{pa 1000 800}%
\special{pa 980 867}%
\special{pa 1000 853}%
\special{pa 1020 867}%
\special{pa 1000 800}%
\special{fp}%
\special{pa 800 800}%
\special{pa 800 400}%
\special{fp}%
\special{sh 1}%
\special{pa 800 400}%
\special{pa 780 467}%
\special{pa 800 453}%
\special{pa 820 467}%
\special{pa 800 400}%
\special{fp}%
\special{pa 1600 800}%
\special{pa 1600 400}%
\special{fp}%
\special{sh 1}%
\special{pa 1600 400}%
\special{pa 1580 467}%
\special{pa 1600 453}%
\special{pa 1620 467}%
\special{pa 1600 400}%
\special{fp}%
\end{picture}}}
\caption{The grim reaper}
\label{grim-reaper}
\end{minipage}
\begin{minipage}{0.49\columnwidth}
\centering
\scalebox{0.7}{{\unitlength 0.1in%
\begin{picture}(38.0000,23.0000)(2.0000,-24.0000)%
\special{pn 20}%
\special{pn 20}%
\special{pa 705 400}%
\special{pa 705 422}%
\special{ip}%
\special{pa 705 449}%
\special{pa 705 471}%
\special{ip}%
\special{pa 705 498}%
\special{pa 705 520}%
\special{ip}%
\special{pa 705 547}%
\special{pa 705 569}%
\special{ip}%
\special{pa 705 596}%
\special{pa 710 897}%
\special{pa 715 1073}%
\special{pa 720 1198}%
\special{pa 725 1295}%
\special{pa 730 1374}%
\special{pa 735 1440}%
\special{pa 740 1498}%
\special{pa 745 1549}%
\special{pa 750 1594}%
\special{pa 755 1635}%
\special{pa 760 1673}%
\special{pa 765 1707}%
\special{pa 770 1739}%
\special{pa 775 1768}%
\special{pa 780 1796}%
\special{pa 790 1846}%
\special{pa 800 1890}%
\special{pa 810 1930}%
\special{pa 820 1966}%
\special{pa 825 1983}%
\special{pa 830 1999}%
\special{pa 840 2029}%
\special{pa 850 2057}%
\special{pa 860 2083}%
\special{pa 870 2107}%
\special{pa 885 2140}%
\special{pa 895 2160}%
\special{pa 915 2196}%
\special{pa 920 2204}%
\special{pa 925 2213}%
\special{pa 930 2220}%
\special{pa 935 2228}%
\special{pa 940 2235}%
\special{pa 945 2243}%
\special{pa 950 2249}%
\special{pa 960 2263}%
\special{pa 980 2287}%
\special{pa 985 2292}%
\special{pa 990 2298}%
\special{pa 1010 2318}%
\special{pa 1015 2322}%
\special{pa 1020 2327}%
\special{pa 1040 2343}%
\special{pa 1045 2346}%
\special{pa 1050 2350}%
\special{pa 1055 2353}%
\special{pa 1060 2357}%
\special{pa 1075 2366}%
\special{pa 1080 2368}%
\special{pa 1090 2374}%
\special{pa 1125 2388}%
\special{pa 1130 2389}%
\special{pa 1135 2391}%
\special{pa 1145 2393}%
\special{pa 1150 2395}%
\special{pa 1160 2397}%
\special{pa 1165 2397}%
\special{pa 1175 2399}%
\special{pa 1180 2399}%
\special{pa 1185 2400}%
\special{pa 1215 2400}%
\special{pa 1220 2399}%
\special{pa 1225 2399}%
\special{pa 1235 2397}%
\special{pa 1240 2397}%
\special{pa 1250 2395}%
\special{pa 1255 2393}%
\special{pa 1265 2391}%
\special{pa 1270 2389}%
\special{pa 1275 2388}%
\special{pa 1310 2374}%
\special{pa 1320 2368}%
\special{pa 1325 2366}%
\special{pa 1340 2357}%
\special{pa 1345 2353}%
\special{pa 1350 2350}%
\special{pa 1355 2346}%
\special{pa 1360 2343}%
\special{pa 1380 2327}%
\special{pa 1385 2322}%
\special{pa 1390 2318}%
\special{pa 1410 2298}%
\special{pa 1415 2292}%
\special{pa 1420 2287}%
\special{pa 1440 2263}%
\special{pa 1450 2249}%
\special{pa 1455 2243}%
\special{pa 1460 2235}%
\special{pa 1465 2228}%
\special{pa 1470 2220}%
\special{pa 1475 2213}%
\special{pa 1480 2204}%
\special{pa 1485 2196}%
\special{pa 1505 2160}%
\special{pa 1515 2140}%
\special{pa 1530 2107}%
\special{pa 1540 2083}%
\special{pa 1550 2057}%
\special{pa 1560 2029}%
\special{pa 1570 1999}%
\special{pa 1575 1983}%
\special{pa 1580 1966}%
\special{pa 1590 1930}%
\special{pa 1600 1890}%
\special{pa 1610 1846}%
\special{pa 1620 1796}%
\special{pa 1625 1768}%
\special{pa 1630 1739}%
\special{pa 1635 1707}%
\special{pa 1640 1673}%
\special{pa 1645 1635}%
\special{pa 1650 1594}%
\special{pa 1655 1549}%
\special{pa 1660 1498}%
\special{pa 1665 1440}%
\special{pa 1670 1374}%
\special{pa 1675 1295}%
\special{pa 1680 1198}%
\special{pa 1685 1073}%
\special{pa 1690 897}%
\special{pa 1695 596}%
\special{pa 1695 400}%
\special{fp}%
\special{pn 20}%
\special{pn 20}%
\special{pa 2505 400}%
\special{pa 2505 420}%
\special{ip}%
\special{pa 2505 446}%
\special{pa 2505 466}%
\special{ip}%
\special{pa 2505 491}%
\special{pa 2505 512}%
\special{ip}%
\special{pa 2505 537}%
\special{pa 2505 557}%
\special{ip}%
\special{pa 2505 583}%
\special{pa 2505 596}%
\special{pa 2505 603}%
\special{ip}%
\special{pa 2506 628}%
\special{pa 2506 649}%
\special{ip}%
\special{pa 2506 674}%
\special{pa 2507 694}%
\special{ip}%
\special{pa 2507 720}%
\special{pa 2507 740}%
\special{ip}%
\special{pa 2508 765}%
\special{pa 2508 786}%
\special{ip}%
\special{pa 2509 811}%
\special{pa 2509 831}%
\special{ip}%
\special{pa 2509 857}%
\special{pa 2510 877}%
\special{ip}%
\special{pa 2510 902}%
\special{pa 2511 922}%
\special{ip}%
\special{pa 2511 948}%
\special{pa 2512 968}%
\special{ip}%
\special{pa 2513 993}%
\special{pa 2513 1014}%
\special{ip}%
\special{pa 2514 1039}%
\special{pa 2515 1059}%
\special{ip}%
\special{pa 2515 1085}%
\special{pa 2516 1105}%
\special{ip}%
\special{pa 2517 1130}%
\special{pa 2518 1151}%
\special{ip}%
\special{pa 2519 1176}%
\special{pa 2520 1196}%
\special{ip}%
\special{pa 2521 1222}%
\special{pa 2522 1242}%
\special{ip}%
\special{pa 2524 1267}%
\special{pa 2525 1287}%
\special{ip}%
\special{pa 2526 1313}%
\special{pa 2527 1333}%
\special{ip}%
\special{pa 2529 1358}%
\special{pa 2530 1374}%
\special{pa 2530 1379}%
\special{ip}%
\special{pa 2532 1404}%
\special{pa 2534 1424}%
\special{ip}%
\special{pa 2536 1449}%
\special{pa 2538 1470}%
\special{ip}%
\special{pa 2540 1495}%
\special{pa 2540 1498}%
\special{pa 2542 1515}%
\special{ip}%
\special{pa 2544 1540}%
\special{pa 2545 1549}%
\special{pa 2546 1560}%
\special{ip}%
\special{pa 2549 1586}%
\special{pa 2550 1594}%
\special{pa 2551 1606}%
\special{ip}%
\special{pa 2555 1631}%
\special{pa 2555 1635}%
\special{pa 2557 1651}%
\special{ip}%
\special{pa 2560 1676}%
\special{pa 2563 1696}%
\special{ip}%
\special{pa 2567 1721}%
\special{pa 2570 1739}%
\special{pa 2570 1741}%
\special{ip}%
\special{pa 2575 1766}%
\special{pa 2575 1768}%
\special{pa 2578 1786}%
\special{ip}%
\special{pa 2583 1811}%
\special{pa 2587 1831}%
\special{ip}%
\special{pa 2592 1856}%
\special{pa 2597 1876}%
\special{ip}%
\special{pa 2603 1901}%
\special{pa 2608 1920}%
\special{ip}%
\special{pa 2614 1945}%
\special{pa 2620 1964}%
\special{ip}%
\special{pa 2627 1989}%
\special{pa 2630 1999}%
\special{pa 2633 2008}%
\special{ip}%
\special{pa 2641 2032}%
\special{pa 2648 2051}%
\special{ip}%
\special{pa 2657 2075}%
\special{pa 2660 2083}%
\special{pa 2664 2094}%
\special{ip}%
\special{pa 2674 2117}%
\special{pa 2683 2135}%
\special{ip}%
\special{pa 2694 2158}%
\special{pa 2695 2160}%
\special{pa 2704 2176}%
\special{ip}%
\special{pa 2716 2198}%
\special{pa 2720 2204}%
\special{pa 2725 2213}%
\special{pa 2727 2215}%
\special{ip}%
\special{pa 2741 2236}%
\special{pa 2745 2243}%
\special{pa 2750 2249}%
\special{pa 2753 2253}%
\special{ip}%
\special{pa 2768 2273}%
\special{pa 2780 2287}%
\special{pa 2781 2288}%
\special{ip}%
\special{pa 2799 2307}%
\special{pa 2810 2318}%
\special{pa 2813 2321}%
\special{ip}%
\special{pa 2833 2337}%
\special{pa 2840 2343}%
\special{pa 2845 2346}%
\special{pa 2849 2349}%
\special{ip}%
\special{pa 2870 2363}%
\special{pa 2875 2366}%
\special{pa 2880 2368}%
\special{pa 2888 2373}%
\special{ip}%
\special{pa 2911 2383}%
\special{pa 2925 2388}%
\special{pa 2930 2389}%
\special{pa 2930 2389}%
\special{ip}%
\special{pa 2955 2396}%
\special{pa 2960 2397}%
\special{pa 2965 2397}%
\special{pa 2975 2399}%
\special{ip}%
\special{pa 3000 2400}%
\special{pa 3015 2400}%
\special{pa 3020 2399}%
\special{pa 3025 2399}%
\special{pa 3035 2397}%
\special{pa 3040 2397}%
\special{pa 3050 2395}%
\special{pa 3055 2393}%
\special{pa 3065 2391}%
\special{pa 3070 2389}%
\special{pa 3075 2388}%
\special{pa 3110 2374}%
\special{pa 3120 2368}%
\special{pa 3125 2366}%
\special{pa 3140 2357}%
\special{pa 3145 2353}%
\special{pa 3150 2350}%
\special{pa 3155 2346}%
\special{pa 3160 2343}%
\special{pa 3180 2327}%
\special{pa 3185 2322}%
\special{pa 3190 2318}%
\special{pa 3210 2298}%
\special{pa 3215 2292}%
\special{pa 3220 2287}%
\special{pa 3240 2263}%
\special{pa 3250 2249}%
\special{pa 3255 2243}%
\special{pa 3260 2235}%
\special{pa 3265 2228}%
\special{pa 3270 2220}%
\special{pa 3275 2213}%
\special{pa 3280 2204}%
\special{pa 3285 2196}%
\special{pa 3305 2160}%
\special{pa 3315 2140}%
\special{pa 3330 2107}%
\special{pa 3340 2083}%
\special{pa 3350 2057}%
\special{pa 3360 2029}%
\special{pa 3370 1999}%
\special{pa 3375 1983}%
\special{pa 3380 1966}%
\special{pa 3390 1930}%
\special{pa 3400 1890}%
\special{pa 3410 1846}%
\special{pa 3420 1796}%
\special{pa 3425 1768}%
\special{pa 3430 1739}%
\special{pa 3435 1707}%
\special{pa 3440 1673}%
\special{pa 3445 1635}%
\special{pa 3450 1594}%
\special{pa 3455 1549}%
\special{pa 3460 1498}%
\special{pa 3465 1440}%
\special{pa 3470 1374}%
\special{pa 3475 1295}%
\special{pa 3480 1198}%
\special{pa 3485 1073}%
\special{pa 3490 897}%
\special{pa 3495 596}%
\special{pa 3495 400}%
\special{fp}%
%
\special{pn 20}%
\special{pa 1700 400}%
\special{pa 3500 400}%
\special{fp}%
%
\special{pn 20}%
\special{pa 710 740}%
\special{pa 1710 740}%
\special{fp}%
%
\special{pn 20}%
\special{pa 1710 740}%
\special{pa 2510 740}%
\special{da 0.070}%
\special{pa 2510 740}%
\special{pa 2510 740}%
\special{da 0.070}%
%
\special{pn 20}%
\special{pa 1200 2400}%
\special{pa 3000 2400}%
\special{fp}%
\special{pn 20}%
\special{pn 20}%
\special{pa 2505 400}%
\special{pa 2505 422}%
\special{ip}%
\special{pa 2505 449}%
\special{pa 2505 471}%
\special{ip}%
\special{pa 2505 498}%
\special{pa 2505 520}%
\special{ip}%
\special{pa 2505 547}%
\special{pa 2505 569}%
\special{ip}%
\special{pn 20}%
\special{pa 2505 596}%
\special{pa 2506 660}%
\special{fp}%
\special{pa 2507 725}%
\special{pa 2508 788}%
\special{fp}%
\special{pa 2509 854}%
\special{pa 2510 897}%
\special{pa 2511 917}%
\special{fp}%
\special{pa 2512 982}%
\special{pa 2514 1046}%
\special{fp}%
\special{pa 2517 1111}%
\special{pa 2519 1175}%
\special{fp}%
\special{pa 2522 1240}%
\special{pa 2525 1295}%
\special{pa 2526 1303}%
\special{fp}%
\special{pa 2530 1368}%
\special{pa 2530 1374}%
\special{pa 2534 1432}%
\special{fp}%
\special{pa 2540 1497}%
\special{pa 2540 1498}%
\special{pa 2545 1549}%
\special{pa 2546 1560}%
\special{fp}%
\special{pa 2554 1625}%
\special{pa 2555 1635}%
\special{pa 2560 1673}%
\special{pa 2562 1688}%
\special{fp}%
\special{pa 2572 1752}%
\special{pa 2575 1768}%
\special{pa 2580 1796}%
\special{pa 2584 1815}%
\special{fp}%
\special{pa 2597 1878}%
\special{pa 2600 1890}%
\special{pa 2610 1930}%
\special{pa 2613 1940}%
\special{fp}%
\special{pa 2631 2003}%
\special{pa 2640 2029}%
\special{pa 2650 2057}%
\special{pa 2652 2063}%
\special{fp}%
\special{pa 2677 2123}%
\special{pa 2685 2140}%
\special{pa 2695 2160}%
\special{pa 2706 2180}%
\special{fp}%
\special{pa 2740 2235}%
\special{pa 2740 2235}%
\special{pa 2745 2243}%
\special{pa 2750 2249}%
\special{pa 2760 2263}%
\special{pa 2779 2285}%
\special{fp}%
\special{pa 2825 2331}%
\special{pa 2840 2343}%
\special{pa 2845 2346}%
\special{pa 2850 2350}%
\special{pa 2855 2353}%
\special{pa 2860 2357}%
\special{pa 2875 2366}%
\special{pa 2877 2367}%
\special{fp}%
\special{pa 2937 2391}%
\special{pa 2945 2393}%
\special{pa 2950 2395}%
\special{pa 2960 2397}%
\special{pa 2965 2397}%
\special{pa 2975 2399}%
\special{pa 2980 2399}%
\special{pa 2985 2400}%
\special{pa 3000 2400}%
\special{fp}%
\special{pn 20}%
\special{pa 3020 2399}%
\special{pa 3025 2399}%
\special{pa 3035 2397}%
\special{pa 3040 2397}%
\special{pa 3045 2396}%
\special{ip}%
\special{pa 3065 2391}%
\special{pa 3065 2391}%
\special{pa 3070 2389}%
\special{pa 3075 2388}%
\special{pa 3089 2383}%
\special{ip}%
\special{pa 3108 2375}%
\special{pa 3110 2374}%
\special{pa 3120 2368}%
\special{pa 3125 2366}%
\special{pa 3130 2363}%
\special{ip}%
\special{pa 3147 2352}%
\special{pa 3150 2350}%
\special{pa 3155 2346}%
\special{pa 3160 2343}%
\special{pa 3167 2337}%
\special{ip}%
\special{pa 3183 2324}%
\special{pa 3185 2322}%
\special{pa 3190 2318}%
\special{pa 3201 2307}%
\special{ip}%
\special{pa 3215 2292}%
\special{pa 3220 2287}%
\special{pa 3232 2273}%
\special{ip}%
\special{pa 3244 2257}%
\special{pa 3250 2249}%
\special{pa 3255 2243}%
\special{pa 3259 2236}%
\special{ip}%
\special{pa 3270 2219}%
\special{pa 3275 2213}%
\special{pa 3280 2204}%
\special{pa 3284 2198}%
\special{ip}%
\special{pa 3294 2180}%
\special{pa 3305 2160}%
\special{pa 3306 2158}%
\special{ip}%
\special{pa 3315 2140}%
\special{pa 3326 2117}%
\special{ip}%
\special{pa 3334 2098}%
\special{pa 3340 2083}%
\special{pa 3343 2075}%
\special{ip}%
\special{pa 3350 2056}%
\special{pa 3359 2032}%
\special{ip}%
\special{pa 3365 2013}%
\special{pa 3370 1999}%
\special{pa 3373 1989}%
\special{ip}%
\special{pa 3379 1969}%
\special{pa 3380 1966}%
\special{pa 3386 1945}%
\special{ip}%
\special{pa 3391 1925}%
\special{pa 3397 1901}%
\special{ip}%
\special{pa 3402 1881}%
\special{pa 3408 1856}%
\special{ip}%
\special{pa 3412 1836}%
\special{pa 3417 1811}%
\special{ip}%
\special{pa 3421 1791}%
\special{pa 3425 1768}%
\special{pa 3425 1766}%
\special{ip}%
\special{pa 3429 1746}%
\special{pa 3430 1739}%
\special{pa 3433 1721}%
\special{ip}%
\special{pa 3436 1701}%
\special{pa 3440 1676}%
\special{ip}%
\special{pa 3442 1656}%
\special{pa 3445 1635}%
\special{pa 3445 1631}%
\special{ip}%
\special{pa 3448 1611}%
\special{pa 3450 1594}%
\special{pa 3451 1586}%
\special{ip}%
\special{pa 3453 1565}%
\special{pa 3455 1549}%
\special{pa 3456 1540}%
\special{ip}%
\special{pa 3458 1520}%
\special{pa 3460 1498}%
\special{pa 3460 1495}%
\special{ip}%
\special{pa 3462 1475}%
\special{pa 3464 1449}%
\special{ip}%
\special{pa 3466 1429}%
\special{pa 3468 1404}%
\special{ip}%
\special{pa 3469 1384}%
\special{pa 3470 1374}%
\special{pa 3471 1358}%
\special{ip}%
\special{pa 3472 1338}%
\special{pa 3474 1313}%
\special{ip}%
\special{pa 3475 1292}%
\special{pa 3476 1267}%
\special{ip}%
\special{pa 3477 1247}%
\special{pa 3479 1222}%
\special{ip}%
\special{pa 3480 1201}%
\special{pa 3480 1198}%
\special{pa 3481 1176}%
\special{ip}%
\special{pa 3482 1156}%
\special{pa 3483 1130}%
\special{ip}%
\special{pa 3484 1110}%
\special{pa 3485 1085}%
\special{ip}%
\special{pa 3485 1064}%
\special{pa 3486 1039}%
\special{ip}%
\special{pa 3487 1019}%
\special{pa 3487 993}%
\special{ip}%
\special{pa 3488 973}%
\special{pa 3489 948}%
\special{ip}%
\special{pa 3489 928}%
\special{pa 3490 902}%
\special{ip}%
\special{pa 3490 882}%
\special{pa 3491 857}%
\special{ip}%
\special{pa 3491 836}%
\special{pa 3491 811}%
\special{ip}%
\special{pa 3492 791}%
\special{pa 3492 765}%
\special{ip}%
\special{pa 3493 745}%
\special{pa 3493 720}%
\special{ip}%
\special{pa 3493 699}%
\special{pa 3494 674}%
\special{ip}%
\special{pa 3494 654}%
\special{pa 3494 628}%
\special{ip}%
\special{pa 3495 608}%
\special{pa 3495 596}%
\special{pa 3495 583}%
\special{ip}%
\special{pa 3495 562}%
\special{pa 3495 537}%
\special{ip}%
\special{pa 3495 517}%
\special{pa 3495 491}%
\special{ip}%
\special{pa 3495 471}%
\special{pa 3495 446}%
\special{ip}%
\special{pa 3495 425}%
\special{pa 3495 400}%
\special{ip}%
%
\special{pn 13}%
\special{pa 800 600}%
\special{pa 800 400}%
\special{fp}%
\special{sh 1}%
\special{pa 800 400}%
\special{pa 780 467}%
\special{pa 800 453}%
\special{pa 820 467}%
\special{pa 800 400}%
\special{fp}%
\special{pa 1300 600}%
\special{pa 1300 400}%
\special{fp}%
\special{sh 1}%
\special{pa 1300 400}%
\special{pa 1280 467}%
\special{pa 1300 453}%
\special{pa 1320 467}%
\special{pa 1300 400}%
\special{fp}%
\special{pa 1900 300}%
\special{pa 1900 100}%
\special{fp}%
\special{sh 1}%
\special{pa 1900 100}%
\special{pa 1880 167}%
\special{pa 1900 153}%
\special{pa 1920 167}%
\special{pa 1900 100}%
\special{fp}%
\special{pa 2400 300}%
\special{pa 2400 100}%
\special{fp}%
\special{sh 1}%
\special{pa 2400 100}%
\special{pa 2380 167}%
\special{pa 2400 153}%
\special{pa 2420 167}%
\special{pa 2400 100}%
\special{fp}%
\special{pa 2900 300}%
\special{pa 2900 100}%
\special{fp}%
\special{sh 1}%
\special{pa 2900 100}%
\special{pa 2880 167}%
\special{pa 2900 153}%
\special{pa 2920 167}%
\special{pa 2900 100}%
\special{fp}%
\special{pa 3400 300}%
\special{pa 3400 100}%
\special{fp}%
\special{sh 1}%
\special{pa 3400 100}%
\special{pa 3380 167}%
\special{pa 3400 153}%
\special{pa 3420 167}%
\special{pa 3400 100}%
\special{fp}%
\end{picture}}}
\caption{The grim hyperplane}
\label{grim-hyperplane}
\end{minipage}
\end{figure}

In this paper, we consider that the case where $N$ is the $n$-dimensional unit sphere $\mathbb{S}^n$ and $u$ is a composition of an isoparametric function on $\mathbb{S}^n$ and some function.
The level sets of the isoparametric functions give compact isoparametric hypersurfaces of $\mathbb{S}^n$.
M\"{u}nzner\cite{M} showed that the number $k$ of distinct principal curvatures of compact isoparametric hypersurfaces of $\mathbb{S}^n$ is $1$, $2$, $3$, $4$ or $6$ by a topological method.
In cases $k=1$, $2$, $3$, Cartan\cite{C} classified the isoparametric hypersurfaces.
The hypersurfaces are $S^{n-1}\subset S^n$ in case $k=1$, $S^k\times S^{n-k-1}\subset S^n$ in case $k=2$ and the tubes over the Veronese surfaces $\mathbb{R}P^2\subset S^4$, $\mathbb{C}P^2\subset S^7$, $\mathbb{Q}P^2\subset S^{13}$, $\mathbb{O}P^2\subset S^{25}$ (i.e., the principal orbits of the isotropy representations of the rank two symmetric spaces $SU(3)/SO(3)$, $(SU(3)\times SU(3))/SU(3)$, $SU(6)/Sp(3)$, $E_6/F_4$) in case $k=3$.
These hypersurfaces are homogeneous.
In case $k=6$, the hypersurfaces are homogeneous by the result of Dorfmeister and Neher\cite{DN} and Miyaoka\cite{Mi}.
The hypersurfaces are the principal orbits of the isotropy representations of $(G_2\times G_2)/G_2$, $G_2/SO(4)$.
In case $k=4$, Ozeki and Takeuchi\cite{OT1,OT2} found that non-homogeneous isoparametric hypersurfaces are constructed as the regular level sets of the restrictions of the Cartin-M\"{u}nzner polynomial functions to the sphere.

In this paper, we obtain the following result.
\begin{theorem}\label{thm:shape-graph}
Let $r$ be an isoparametric function on $\mathbb{S}^n$~$(n\ge2)$ and $V$ be a $C^{\infty}$-$function$ on an interval $J\subset r(\mathbb{S}^n)$.
If the mean curvature flow starting from the graph of the function $u=(V\circ r)\vert_{r^{-1}(J)}$ is a translating soliton, the shape of the graph of $V$ is like one of  those defined in {\rm Figures \ref{Vex1}}$-${\rm \ref{Vex7}}.
The real number $R\in (-1,1)$ in {\rm Figures \ref{Vex1}}$-${\rm \ref{Vex7}} is given by
\begin{equation*}
R:=
\begin{cases}
0\hspace{2.125cm}(k=1,3,6)\\
\displaystyle -1+\frac{km}{n-1}\quad(k=2,4),
\end{cases}
\end{equation*}
where $k$ is the number of distinct principal curvatures of the compact isoparametric hypersurface defined by the regular level set for $r$ and $m$ is the multiplicity of the smallest principal curvature of the isoparametric hypersurface.
\end{theorem}

\begin{figure}[H]
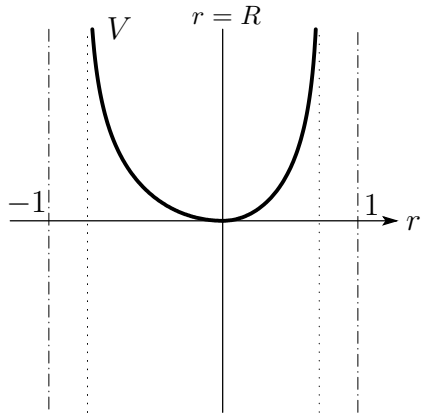

\centering
\scalebox{1.0}{{\unitlength 0.1in%
}}
\caption{The graph of $V$ (Type I)}
\label{Vex1}
\end{figure}

\begin{figure}[H]
\centering
\begin{minipage}{0.49\columnwidth}
\centering
\scalebox{1.0}{{\unitlength 0.1in%
%
}}
\caption{The graph of $V$ (Type II)}
\label{Vex2}
\end{minipage}
\begin{minipage}{0.49\columnwidth}
\centering
\scalebox{1.0}{{\unitlength 0.1in%
%
}}
\caption{The graph of $V$ (Type III)}
\label{Vex3}
\end{minipage}
\end{figure}

\begin{figure}[H]
\centering
\begin{minipage}{0.49\columnwidth}
\centering
\scalebox{1.0}{{\unitlength 0.1in%
%
}}
\caption{The graph of $V$ (Type IV)}
\label{Vex4}
\end{minipage}
\begin{minipage}{0.49\columnwidth}
\centering
\scalebox{1.0}{{\unitlength 0.1in%
%
}}
\caption{The graph of $V$ (Type V)}
\label{Vex5}
\end{minipage}
\end{figure}

\begin{figure}[H]
\centering
\begin{minipage}{0.49\columnwidth}
\centering
\scalebox{1.0}{{\unitlength 0.1in%
%
}}
\caption{The graph of $V$ (Type VI)}
\label{Vex6}
\end{minipage}
\begin{minipage}{0.49\columnwidth}
\centering
\scalebox{1.0}{{\unitlength 0.1in%
%
}}
\caption{The graph of $V$ (Type VII)}
\label{Vex7}
\end{minipage}
\end{figure}


The function $u=(V\circ r)\vert_{r^{-1}(J)}$ in Theorem \ref{thm:shape-graph} is constant on the level set of $r$ and its behavior on the normal direction for the level set of $r$ is a little understood from the behavior of $V$ in {\rm Figures \ref{Vex1}}$-${\rm \ref{Vex7}}.
In the last section, we investigate the domain of the function $u$ in Theorem \ref{thm:shape-graph}.

\section{Basic facts}

Let $g$ be a Riemannian metric of a $n$-dimensional Riemannian manifold $N$ and $u:M\to\mathbb{R}$ be a function on a domain $M\subset N$.
Define the immersion $f$ of $M$ into the product Riemannian manifold $N\times\mathbb{R}$ by $f(x)=(x,u(x)),~x\in M$.
Denote the graph of $u$ by $\Gamma$ and the mean curvature vector field of $f$ by $H$.
Further, we assume that $X$ is a Killing vector field on $N\times\mathbb{R}$ and $\{\phi_t\}_{t\in\mathbb{R}}$ is the one-parameter transformation associated to $X$ on $N\times\mathbb{R}$.
Then, we have the following lemma about the soliton of the mean curvature flow.
\begin{lemma}\label{lemma:soliton-condition}
If the mean curvature flow starting from $\Gamma$ is $X$-soliton, $f$ satisfies
\begin{equation}
\left(X\circ f\right)^{\bot_f}=H.\label{eq:soliton0}
\end{equation}
Conversely, if $f$ satisfies {\rm (\ref{eq:soliton0})}, the familly of the images $\{M_t\}_{t\in\mathbb{R}}$ defined by $f_t=\phi_t\circ f$ and $M_t=f_t(M)$ is the $X$-soliton.
\end{lemma}
\begin{proof}
According to Hungerb\"{u}hler and Smoczyk\cite{HS}, we find the first half of the lemma.
For the second half of the lemma, since $\phi_t$'s are isometries and $f$ satisfies {\rm (\ref{eq:soliton0})}, we find that $f_t=\phi_t\circ f$ satisfies
\begin{align*}
\left(\frac{\partial f_t}{\partial t}\right)^{\bot_{f_t}}-H_t&=d\phi_t((X\circ f)^{\bot_f}-H)\\
&=0,
\end{align*}
and $\{f_t\}_{t\in \mathbb{R}}$ satisfies (\ref{eq:mcf}).
Therefore, $\{M_t\}_{t\in\mathbb{R}}$ is the mean curvature flow.
Further, by $\phi_t^{-1}\circ f_t=f$, it turns out that $f_t$ satisfies (\ref{eq:mcf-soliton}).
So, $\{M_t\}_{t\in\mathbb{R}}$ is the translating soliton.
\end{proof}

Let $\nabla$ and div be the gradient and divergence with respect to $g$ respectively.
For Lemma \ref{lemma:soliton-condition}, considering the case where an $X$-soliton is a translating soliton, the following lemma is derived.
\begin{lemma}\label{lemma:graph-soliton-condition}
If the mean curvature flow starting from $\Gamma$ is a translating soliton, $u$ satisfies
\begin{equation}
\sqrt{1+\|\nabla u\|^2}~{\rm div}\left(\frac{\nabla u}{\sqrt{1+\|\nabla u\|^2}}\right)=1.\label{eq:graph-soliton}
\end{equation}
Conversely, if $u$ satisfies {\rm (\ref{eq:graph-soliton})}, the family of the images $\{M_t\}_{t\in\mathbb{R}}$ definded by $f_t(x)=(x,u(x)+t),~x\in M$ and $M_t=f_t(M)$ is a translating soliton.
\end{lemma}
\begin{proof}
Let $(x^1,\cdots ,x^n,s)$ be local cordinates of $N\times\mathbb{R}$.
Define the Killing vector $X=(0,1)\in T(N\times\mathbb{R})=TN\oplus T\mathbb{R}$.
By $f(x)=(x,u(x)),~x\in M$ and $X=\frac{\partial}{\partial s}$, we find
\begin{align*}
&\left(X\circ f\right)^{\bot_f}=\frac{\partial}{\partial s}-\frac{1}{1+\|\nabla u\|^2}df(\nabla u)\\
&H=\sqrt{1+\|\nabla u\|^2}~{\rm div}\left(\frac{\nabla u}{\sqrt{1+\|\nabla u\|^2}}\right)\left(\frac{\partial}{\partial s}-\frac{1}{1+\|\nabla u\|^2}df(\nabla u)\right).
\end{align*}
Therefore, we obtain that (\ref{lemma:soliton-condition}) and (\ref{eq:graph-soliton}) are equivalent in this case.
\end{proof}

Next, we consider the case where $u$ is a composition of an isoparametric function and some function.
Let $\Delta$ be the Laplacian with respect to $g$.
A non-constant $C^{\infty}$-function $r:N\to\mathbb{R}$ is called an isoparametric function if there exist $C^{\infty}$-functions $\alpha,\beta$ such that
\begin{equation*}
\begin{cases}
\|\nabla r\|^2=\alpha\circ r\\
\Delta r=\beta\circ r.
\end{cases}
\end{equation*}
Further, the level set of $r$ with respect to a regular value is called an isoparametric hypersurface.

In case where $N$ is the $n$-dimensional unit spere $\mathbb{S}^n$, M\"{u}nzner\cite{M} showed the following theorem for an isoparametric function on $\mathbb{S}^n$.
\begin{theorem}
{\rm (M\"{u}nzner\cite{M})}
(i) An isoparametric function $r$ on $\mathbb{S}^n$ is a restriction to $\mathbb{S}^n$ of a homogeneous polynomial $h:\mathbb{R}^{n+1}\to\mathbb{R}$ which satisfies
\begin{equation}\label{eq:CMpol}
\begin{cases}
\displaystyle\|(\nabla^{\mathbb{R}}h)_x\|^2=k^2\vert x\vert^{2k-2}\\
\displaystyle(\Delta_{\mathbb{R}}h)_x=\frac{m_2-m_1}{2}k^2\vert x\vert^{k-2}
\end{cases}\quad (x\in\mathbb{R}^{n+1}),
\end{equation}
where $\vert\bullet\vert$ is the Euclidean norm and $\nabla^{\mathbb{R}}$ and $\Delta_{\mathbb{R}}$ are the gradient and Laplacian for the Euclidean space $\mathbb{R}^n$.
Here, we assume that the isoparametric hypersurface defined by the level set of $r$ has $k$ distinct principal curvatures $\lambda_1>\cdots>\lambda_k$ with respective multiplicities $m_1,\cdots,m_k$.

(ii) The above natural number $k$ is $1$, $2$, $3$, $4$ or $6$.
\end{theorem}
\begin{remark}\label{remark:multi}
According to M\"{u}nzner\cite{M,M2}, we find the following two facts.

(i) If $k=1$, $3$, $6$ , then the mulitiplicities are equal.
If $k=2$, $4$, then there are at most two distinct multiplicities $m_1$, $m_2$.

(ii) By {\rm (\ref{eq:CMpol})}, we obtain
\begin{equation}\label{isopara-sphere}
\begin{cases}
\|\nabla r\|^2=k^2(1-r^2)\\
\displaystyle\Delta r=\frac{m_2-m_1}{2}k^2-k(n+k-1)r.
\end{cases}
\end{equation}
From the first equation of {\rm (\ref{isopara-sphere})}, we find that $r(\mathbb{S}^n)=[-1,1]$.
\end{remark}

For Lemma \ref{lemma:graph-soliton-condition}, considering the case where  $u$ is the composition of the isoparametric function and some function, the following lemma is derived.
\begin{lemma}
Let $r:N\to\mathbb{R}$ be an isoparametric function on $N$.
If  the mean curvature flow starting from $\Gamma$ is a translating soliton and there if exists a $C^{\infty}$-function $V$ on $r(M)$ such that $u=(V\circ r)\vert_M$, the function $V$ satisfies
\begin{equation}
2\alpha V''-\alpha(\alpha'-2\beta) V'^3-2\alpha V'^2+2\beta V'-2=0,\label{eq:isopara-soliton}
\end{equation}
where $'$ denotes derivative on $r(M)$ and $\alpha,\beta$ are $C^{\infty}-$functions which satisfy $\|\nabla r\|^2=\alpha\circ r,~\Delta r=\beta\circ r$.\label{eq:isopara-graph-soliton}
Conversely, if $V$ satisfies {\rm (\ref{eq:isopara-soliton})}, the family of the images $\{M_t\}_{t\in\mathbb{R}}$ defined by $f_t(x)=(x,(V\circ r)(x)+t),~x\in M$ and $M_t=f_t(M)$ is the translating soliton.
\end{lemma}
\begin{proof}
For the left side of (\ref{eq:graph-soliton}), we have
\begin{equation*}
\sqrt{1+\|\nabla u\|^2}~{\rm div}\left(\frac{\nabla u}{\sqrt{1+\|\nabla u\|^2}}\right)=\Delta u-\frac{1}{2(1+\|\nabla u\|^2)}\nabla u(\|\nabla u\|^2).
\end{equation*}
By $u=V\circ r$, we find
\begin{equation*}
\begin{split}
\|\nabla u\|^2&=\left(\alpha V'^2\right)\circ r,\\
\nabla u(\|\nabla u\|^2)&=\left(\alpha V'^2\left(2\alpha V''+\alpha'V'\right)\right)\circ r,\\
\Delta u&=\left(\alpha V''+\beta V'\right)\circ r.
\end{split}
\end{equation*}
Therefore, (\ref{eq:graph-soliton}) is reduced to the following equation
\begin{equation*}
\frac{\alpha V''}{1+\alpha V'^2}+\beta V'-\frac{\alpha\alpha'V'^3}{2\left(1+\alpha V'^2\right)}=1.
\end{equation*}
By this equation, we obtain (\ref{eq:isopara-soliton}).
\end{proof}

\section{Proof of Theorem \ref{thm:shape-graph}}

In this section, we assume that $N$ is the $n$-dimensional unit sphere $\mathbb{S}^n$ ($n\ge2$) and  $u=(V\circ r)\vert_{r^{-1}(J)}$ with an isoparametric function $r$ on $\mathbb{S}^n$ and a $C^{\infty}$-function $V$ on interval $J\subset r(\mathbb{S}^n)=[-1,1]$.
By (\ref{isopara-sphere}), substituting $\alpha(r)=k^2(1-r^2)$ and $\beta(r)=\frac{m_2-m_1}{2}k^2-k(n+k-1)r$ for (\ref{eq:isopara-soliton}), we obtain
\begin{align}
V''(r)=&k((n-1)r-\frac{m_2-m_1}{2}k)V'(r)^3+V'(r)^2\nonumber\\
&+\frac{((n+k-1)r-\frac{m_2-m_1}{2}k)}{k(1-r^2)}V'(r)+\frac{1}{k^2(1-r^2)},~~~r\in (-1,1).\label{eq:isopara-soliton-sphere}
\end{align}
The local existence of the solution $V$ of (\ref{eq:isopara-soliton-sphere}) is clear.
By Remark \ref{remark:multi} (i), we find
\begin{equation*}
m_2-m_1=
\begin{cases}
0\quad \hspace{1.9cm}(k=1,3,6)\\
2(m_2-\frac{n-1}{k})\quad(k=2,4).
\end{cases}
\end{equation*}
Therefore, (\ref{eq:isopara-soliton-sphere}) is reduced to
\begin{align}
V''(r)=&k((n-1)(r-R))V'(r)^3+V'(r)^2\nonumber\\
&+\frac{((n+k-1)r-(n-1)R)}{k(1-r^2)}V'(r)+\frac{1}{k^2(1-r^2)},~~~r\in (-1,1).\label{eq:isopara-soliton-sphere2}
\end{align}
Here, $R\in(-1,1)$ is the constant defined by
\begin{equation*}
R:=
\begin{cases}
0\hspace{2.125cm}(k=1,3,6)\\
\displaystyle -1+\frac{km_2}{n-1}\quad(k=2,4),
\end{cases}
\end{equation*}
and when $k=2,4$, $m_2$ is equal to the multiplicity of the smallest principal curvature of the isoparametric hypersurface defined by the level set of $r$.
To prove Theorem \ref{thm:shape-graph}, we consider the graph of the solution $V$ of (\ref{eq:isopara-soliton-sphere2}).
Define $\psi(r)=k\sqrt{1-r^2}V'(r)$.
The equation (\ref{eq:isopara-soliton-sphere2}) is reduced to
\begin{equation}
\psi'(r)=\frac{1}{k(1-r^2)}\left(\psi(r)^2+1\right)\left((n-1)(r-R)\psi(r)+\sqrt{1-r^2}\right).\label{eq:isopara-soliton-sphere3}
\end{equation}
Therefore, we consider the behavior of the solution $\psi$ of (\ref{eq:isopara-soliton-sphere3}).
Define $\eta(r)=-\frac{\sqrt{1-r^2}}{(n-1)(r-R)}$.
Then, the following lemma holds clearly.
\begin{lemma}\label{thm:sign}
~
\begin{itemize}
\item[(i)] When $r\in(R,1)$ {\rm :}
\begin{itemize}
\item [(a)] if $\psi(r)>\eta(r)$, then $\psi'(r)>0$,
\item [(b)] if $\psi(r)=\eta(r)$, then $\psi'(r)=0$,
\item [(c)] if $\psi(r)<\eta(r)$, then $\psi'(r)<0$.
\end{itemize}
\item[(ii)] When $r\in(-1,R)$ {\rm :}
\begin{itemize}
\item [(a)] if $\psi(r)<\eta(r)$, then $\psi'(r)>0$,
\item [(b)] if $\psi(r)=\eta(r)$, then $\psi'(r)=0$,
\item [(c)] if $\psi(r)>\eta(r)$, then $\psi'(r)<0$.
\end{itemize}
\item[(iii)] When $r=R$ or $\psi(r)=0$ {\rm :} $\psi'(r)>0$.
\end{itemize}
\end{lemma}
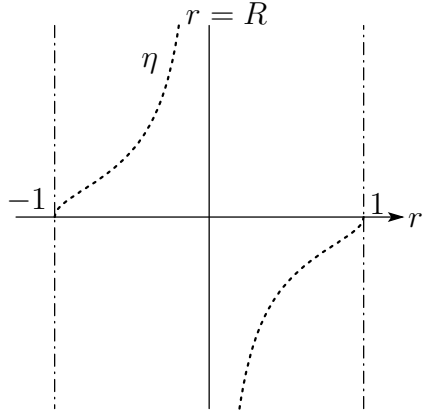
\begin{figure}[H]
\centering
\scalebox{1.0}{{\unitlength 0.1in%
\begin{picture}(20.9000,21.4000)(1.4000,-22.0000)%
\special{pn 13}%
\special{pn 13}%
\special{pa 400 1200}%
\special{pa 403 1187}%
\special{fp}%
\special{pa 417 1158}%
\special{pa 426 1148}%
\special{fp}%
\special{pa 449 1127}%
\special{pa 460 1118}%
\special{fp}%
\special{pa 486 1099}%
\special{pa 497 1092}%
\special{fp}%
\special{pa 524 1074}%
\special{pa 534 1066}%
\special{fp}%
\special{pa 562 1049}%
\special{pa 573 1041}%
\special{fp}%
\special{pa 600 1024}%
\special{pa 611 1017}%
\special{fp}%
\special{pa 637 998}%
\special{pa 648 990}%
\special{fp}%
\special{pa 674 971}%
\special{pa 684 963}%
\special{fp}%
\special{pa 709 943}%
\special{pa 719 934}%
\special{fp}%
\special{pa 743 912}%
\special{pa 753 903}%
\special{fp}%
\special{pa 776 880}%
\special{pa 785 870}%
\special{fp}%
\special{pa 806 846}%
\special{pa 815 836}%
\special{fp}%
\special{pa 835 811}%
\special{pa 842 800}%
\special{fp}%
\special{pa 861 773}%
\special{pa 868 762}%
\special{fp}%
\special{pa 884 734}%
\special{pa 891 723}%
\special{fp}%
\special{pa 906 694}%
\special{pa 912 682}%
\special{fp}%
\special{pa 925 653}%
\special{pa 931 641}%
\special{fp}%
\special{pa 943 611}%
\special{pa 948 598}%
\special{fp}%
\special{pa 959 568}%
\special{pa 963 555}%
\special{fp}%
\special{pa 973 525}%
\special{pa 977 512}%
\special{fp}%
\special{pa 986 481}%
\special{pa 989 468}%
\special{fp}%
\special{pa 997 437}%
\special{pa 1000 424}%
\special{fp}%
\special{pa 1008 392}%
\special{pa 1011 379}%
\special{fp}%
\special{pa 1017 348}%
\special{pa 1020 335}%
\special{fp}%
\special{pa 1026 303}%
\special{pa 1028 290}%
\special{fp}%
\special{pa 1034 258}%
\special{pa 1036 245}%
\special{fp}%
\special{pa 1041 213}%
\special{pa 1043 200}%
\special{fp}%
\special{pn 13}%
\special{pa 1357 2200}%
\special{pa 1359 2187}%
\special{fp}%
\special{pa 1364 2155}%
\special{pa 1366 2142}%
\special{fp}%
\special{pa 1372 2110}%
\special{pa 1374 2097}%
\special{fp}%
\special{pa 1380 2065}%
\special{pa 1383 2052}%
\special{fp}%
\special{pa 1389 2021}%
\special{pa 1392 2008}%
\special{fp}%
\special{pa 1400 1976}%
\special{pa 1403 1963}%
\special{fp}%
\special{pa 1411 1932}%
\special{pa 1414 1919}%
\special{fp}%
\special{pa 1423 1888}%
\special{pa 1427 1875}%
\special{fp}%
\special{pa 1437 1845}%
\special{pa 1441 1832}%
\special{fp}%
\special{pa 1452 1802}%
\special{pa 1457 1789}%
\special{fp}%
\special{pa 1469 1759}%
\special{pa 1475 1747}%
\special{fp}%
\special{pa 1488 1718}%
\special{pa 1494 1706}%
\special{fp}%
\special{pa 1509 1677}%
\special{pa 1516 1666}%
\special{fp}%
\special{pa 1532 1638}%
\special{pa 1539 1627}%
\special{fp}%
\special{pa 1558 1600}%
\special{pa 1565 1589}%
\special{fp}%
\special{pa 1585 1564}%
\special{pa 1594 1554}%
\special{fp}%
\special{pa 1615 1530}%
\special{pa 1624 1520}%
\special{fp}%
\special{pa 1647 1497}%
\special{pa 1657 1488}%
\special{fp}%
\special{pa 1681 1466}%
\special{pa 1691 1457}%
\special{fp}%
\special{pa 1716 1437}%
\special{pa 1726 1429}%
\special{fp}%
\special{pa 1752 1410}%
\special{pa 1763 1402}%
\special{fp}%
\special{pa 1789 1383}%
\special{pa 1800 1376}%
\special{fp}%
\special{pa 1827 1359}%
\special{pa 1838 1351}%
\special{fp}%
\special{pa 1866 1334}%
\special{pa 1876 1326}%
\special{fp}%
\special{pa 1903 1308}%
\special{pa 1914 1301}%
\special{fp}%
\special{pa 1940 1282}%
\special{pa 1951 1273}%
\special{fp}%
\special{pa 1974 1252}%
\special{pa 1983 1242}%
\special{fp}%
\special{pa 1997 1213}%
\special{pa 2000 1200}%
\special{fp}%
%
\special{pn 8}%
\special{pn 8}%
\special{pa 400 200}%
\special{pa 400 263}%
\special{fp}%
\special{pa 400 293}%
\special{pa 400 302}%
\special{fp}%
\special{pa 400 332}%
\special{pa 400 395}%
\special{fp}%
\special{pa 400 426}%
\special{pa 400 434}%
\special{fp}%
\special{pa 400 464}%
\special{pa 400 527}%
\special{fp}%
\special{pa 400 558}%
\special{pa 400 566}%
\special{fp}%
\special{pa 400 596}%
\special{pa 400 659}%
\special{fp}%
\special{pa 400 690}%
\special{pa 400 698}%
\special{fp}%
\special{pa 400 728}%
\special{pa 400 791}%
\special{fp}%
\special{pa 400 821}%
\special{pa 400 829}%
\special{fp}%
\special{pa 400 860}%
\special{pa 400 923}%
\special{fp}%
\special{pa 400 953}%
\special{pa 400 961}%
\special{fp}%
\special{pa 400 992}%
\special{pa 400 1055}%
\special{fp}%
\special{pa 400 1085}%
\special{pa 400 1093}%
\special{fp}%
\special{pa 400 1124}%
\special{pa 400 1187}%
\special{fp}%
\special{pa 400 1217}%
\special{pa 400 1225}%
\special{fp}%
\special{pa 400 1256}%
\special{pa 400 1319}%
\special{fp}%
\special{pa 400 1349}%
\special{pa 400 1357}%
\special{fp}%
\special{pa 400 1388}%
\special{pa 400 1451}%
\special{fp}%
\special{pa 400 1481}%
\special{pa 400 1489}%
\special{fp}%
\special{pa 400 1520}%
\special{pa 400 1583}%
\special{fp}%
\special{pa 400 1613}%
\special{pa 400 1621}%
\special{fp}%
\special{pa 400 1652}%
\special{pa 400 1715}%
\special{fp}%
\special{pa 400 1745}%
\special{pa 400 1753}%
\special{fp}%
\special{pa 400 1784}%
\special{pa 400 1847}%
\special{fp}%
\special{pa 400 1877}%
\special{pa 400 1885}%
\special{fp}%
\special{pa 400 1916}%
\special{pa 400 1979}%
\special{fp}%
\special{pa 400 2009}%
\special{pa 400 2017}%
\special{fp}%
\special{pa 400 2048}%
\special{pa 400 2111}%
\special{fp}%
\special{pa 400 2141}%
\special{pa 400 2149}%
\special{fp}%
\special{pa 400 2180}%
\special{pa 400 2200}%
\special{fp}%
\special{pn 8}%
\special{pa 2000 2200}%
\special{pa 2000 2137}%
\special{fp}%
\special{pa 2000 2107}%
\special{pa 2000 2098}%
\special{fp}%
\special{pa 2000 2068}%
\special{pa 2000 2005}%
\special{fp}%
\special{pa 2000 1974}%
\special{pa 2000 1966}%
\special{fp}%
\special{pa 2000 1936}%
\special{pa 2000 1873}%
\special{fp}%
\special{pa 2000 1842}%
\special{pa 2000 1834}%
\special{fp}%
\special{pa 2000 1804}%
\special{pa 2000 1741}%
\special{fp}%
\special{pa 2000 1710}%
\special{pa 2000 1702}%
\special{fp}%
\special{pa 2000 1672}%
\special{pa 2000 1609}%
\special{fp}%
\special{pa 2000 1579}%
\special{pa 2000 1571}%
\special{fp}%
\special{pa 2000 1540}%
\special{pa 2000 1477}%
\special{fp}%
\special{pa 2000 1447}%
\special{pa 2000 1439}%
\special{fp}%
\special{pa 2000 1408}%
\special{pa 2000 1345}%
\special{fp}%
\special{pa 2000 1315}%
\special{pa 2000 1307}%
\special{fp}%
\special{pa 2000 1276}%
\special{pa 2000 1213}%
\special{fp}%
\special{pa 2000 1183}%
\special{pa 2000 1175}%
\special{fp}%
\special{pa 2000 1144}%
\special{pa 2000 1081}%
\special{fp}%
\special{pa 2000 1051}%
\special{pa 2000 1043}%
\special{fp}%
\special{pa 2000 1012}%
\special{pa 2000 949}%
\special{fp}%
\special{pa 2000 919}%
\special{pa 2000 911}%
\special{fp}%
\special{pa 2000 880}%
\special{pa 2000 817}%
\special{fp}%
\special{pa 2000 787}%
\special{pa 2000 779}%
\special{fp}%
\special{pa 2000 748}%
\special{pa 2000 685}%
\special{fp}%
\special{pa 2000 655}%
\special{pa 2000 647}%
\special{fp}%
\special{pa 2000 616}%
\special{pa 2000 553}%
\special{fp}%
\special{pa 2000 523}%
\special{pa 2000 515}%
\special{fp}%
\special{pa 2000 484}%
\special{pa 2000 421}%
\special{fp}%
\special{pa 2000 391}%
\special{pa 2000 383}%
\special{fp}%
\special{pa 2000 352}%
\special{pa 2000 289}%
\special{fp}%
\special{pa 2000 259}%
\special{pa 2000 251}%
\special{fp}%
\special{pa 2000 220}%
\special{pa 2000 200}%
\special{fp}%
\put(22.3000,-12.5000){\makebox(0,0)[lb]{$r$}}%
\put(20.3000,-11.8000){\makebox(0,0)[lb]{$1$}}%
\put(3.6000,-11.7000){\makebox(0,0)[rb]{$-1$}}%
\put(8.5000,-4.5000){\makebox(0,0)[lb]{$\eta$}}%
%
\special{pn 8}%
\special{pa 200 1200}%
\special{pa 2200 1200}%
\special{fp}%
\special{sh 1}%
\special{pa 2200 1200}%
\special{pa 2133 1180}%
\special{pa 2147 1200}%
\special{pa 2133 1220}%
\special{pa 2200 1200}%
\special{fp}%
%
\special{pn 8}%
\special{pa 1200 2200}%
\special{pa 1200 200}%
\special{fp}%
\put(10.8000,-1.9000){\makebox(0,0)[lb]{$r=R$}}%
\end{picture}}}
\caption{The graph of $\eta$}
\label{eta}
\end{figure}

For the shape of $\psi$ in the case where $\psi>0$, we obtain the following lemmata.

\begin{lemma}\label{thm:psi-shape1}
If there exists $r_0\in(R,1)$ with $\psi(r_0)>0$, there exists $r_1\in(r_0,1)$ such that 
\begin{equation*}
\lim_{r\uparrow r_1}\psi(r)=+\infty.
\end{equation*}
\end{lemma}

\begin{proof}
For all $r\in(r_0,1)$, we find $\psi'(r)>0$ and $\psi(r)>0$.
Also, we have
\begin{align*}
\psi'(r)&=\frac{1}{k(1-r^2)}\Bigl(\psi(r)^2+1\Bigl)\Bigl((n-1)(r-R)\psi(r)+\sqrt{1-r^2}\Bigl)\\
&>\frac{(n-1)(r-R)}{k(1-r^2)}\psi(r)^3.
\end{align*} 
Therefore, we find
\begin{equation*}
\frac{\psi'(r)}{\psi(r)^3}>\frac{(n-1)(r-R)}{k(1-r^2)}.
\end{equation*}
By integrating from $r_0$ to $r$, we have
\begin{align*}
\frac{1}{\psi(r)^2}<\frac{(n-1)}{k}&\log{(1-r^2)}+\frac{(n-1)R}{k}\log{\frac{1+r}{1-r}}\\
&-\frac{(n-1)}{k}\log{(1-r_0^2)}-\frac{(n-1)R}{k}\log{\frac{1+r_0}{1-r_0}}+\frac{1}{\psi(r_0)^2}=:h_1(r).
\end{align*}
Here, $h_1$ is decreasing on $(r_0,1)$ and 
\begin{equation*}
h_1(r_0)=\frac{1}{\psi(r_0)^2}>0,\quad \lim_{r\uparrow 1}h_1(r)=-\infty.
\end{equation*}
Therefore, there exists $\overline{r}_1\in(r_0,1)$ with $h_1(\overline{r}_1)=0$ and
\begin{equation*}
\psi(r)>\frac{1}{\sqrt{h_1(r)}}\rightarrow+\infty\quad(~r~\uparrow~\overline{r}_1~).
\end{equation*}
Then, we obtain the statement of this lemma.
\end{proof}

\begin{figure}[H]
\centering
\scalebox{1.0}{{\unitlength 0.1in%
\begin{picture}(20.9000,21.5000)(1.4000,-22.0000)%
\special{pn 20}%
\special{pn 20}%
\special{pa 590 2200}%
\special{pa 593 2180}%
\special{ip}%
\special{pa 598 2155}%
\special{pa 601 2136}%
\special{ip}%
\special{pa 605 2111}%
\special{pa 609 2091}%
\special{ip}%
\special{pa 614 2067}%
\special{pa 615 2060}%
\special{pa 618 2047}%
\special{ip}%
\special{pa 623 2022}%
\special{pa 627 2003}%
\special{ip}%
\special{pa 632 1978}%
\special{pa 636 1959}%
\special{ip}%
\special{pa 642 1934}%
\special{pa 647 1915}%
\special{ip}%
\special{pa 653 1890}%
\special{pa 655 1883}%
\special{pa 658 1871}%
\special{ip}%
\special{pa 665 1847}%
\special{pa 665 1846}%
\special{pa 670 1827}%
\special{ip}%
\special{pa 678 1803}%
\special{pa 684 1784}%
\special{ip}%
\special{pa 692 1760}%
\special{pa 698 1741}%
\special{ip}%
\special{pa 706 1718}%
\special{pa 713 1699}%
\special{ip}%
\special{pa 723 1676}%
\special{pa 730 1657}%
\special{ip}%
\special{pa 740 1634}%
\special{pa 745 1624}%
\special{pa 749 1616}%
\special{ip}%
\special{pa 760 1593}%
\special{pa 765 1583}%
\special{pa 769 1576}%
\special{ip}%
\special{pa 781 1554}%
\special{pa 790 1538}%
\special{pa 791 1536}%
\special{ip}%
\special{pa 804 1515}%
\special{pa 810 1506}%
\special{pa 815 1499}%
\special{pa 815 1498}%
\special{ip}%
\special{pa 830 1478}%
\special{pa 830 1477}%
\special{pa 835 1471}%
\special{pa 842 1462}%
\special{ip}%
\special{pa 857 1442}%
\special{pa 870 1427}%
\special{pa 870 1427}%
\special{ip}%
\special{pa 887 1408}%
\special{pa 890 1405}%
\special{pa 901 1394}%
\special{ip}%
\special{pa 919 1377}%
\special{pa 925 1371}%
\special{pa 935 1363}%
\special{ip}%
\special{pa 954 1347}%
\special{pa 969 1335}%
\special{ip}%
\special{pa 989 1320}%
\special{pa 990 1319}%
\special{pa 995 1316}%
\special{pa 1000 1312}%
\special{pa 1005 1309}%
\special{pa 1006 1308}%
\special{ip}%
\special{pa 1027 1295}%
\special{pa 1030 1292}%
\special{pa 1044 1284}%
\special{ip}%
\special{pa 1065 1271}%
\special{pa 1075 1265}%
\special{pa 1080 1263}%
\special{pa 1083 1261}%
\special{ip}%
\special{pa 1105 1249}%
\special{pa 1105 1249}%
\special{pa 1115 1243}%
\special{pa 1120 1241}%
\special{pa 1122 1240}%
\special{ip}%
\special{pa 1145 1228}%
\special{pa 1145 1228}%
\special{pa 1150 1225}%
\special{pa 1155 1223}%
\special{pa 1160 1220}%
\special{pa 1162 1219}%
\special{ip}%
\special{pa 1185 1208}%
\special{pa 1185 1208}%
\special{pa 1190 1205}%
\special{pa 1195 1203}%
\special{pa 1202 1199}%
\special{ip}%
\special{pa 1225 1187}%
\special{pa 1225 1187}%
\special{pa 1230 1185}%
\special{pa 1235 1182}%
\special{pa 1240 1180}%
\special{pa 1243 1178}%
\special{ip}%
\special{pa 1265 1167}%
\special{pa 1265 1167}%
\special{pa 1270 1165}%
\special{pa 1280 1159}%
\special{pa 1283 1158}%
\special{ip}%
\special{pa 1305 1146}%
\special{pa 1320 1137}%
\special{pa 1322 1136}%
\special{ip}%
\special{pa 1344 1124}%
\special{pa 1361 1113}%
\special{ip}%
\special{pa 1382 1100}%
\special{pa 1390 1095}%
\special{pa 1395 1091}%
\special{pa 1399 1089}%
\special{ip}%
\special{pa 1419 1074}%
\special{pa 1420 1073}%
\special{pa 1425 1070}%
\special{pa 1435 1062}%
\special{ip}%
\special{pa 1455 1046}%
\special{pa 1460 1042}%
\special{pa 1465 1037}%
\special{pa 1470 1033}%
\special{ip}%
\special{pa 1488 1016}%
\special{pa 1490 1015}%
\special{pa 1503 1002}%
\special{ip}%
\special{pa 1520 984}%
\special{pa 1520 984}%
\special{ip}%
\special{pa 1520 984}%
\special{pa 1525 978}%
\special{pa 1530 973}%
\special{pa 1555 943}%
\special{pa 1565 929}%
\special{pa 1570 923}%
\special{pa 1580 909}%
\special{pa 1585 901}%
\special{pa 1590 894}%
\special{pa 1610 862}%
\special{pa 1635 817}%
\special{pa 1645 797}%
\special{pa 1650 786}%
\special{pa 1655 776}%
\special{pa 1665 754}%
\special{pa 1680 718}%
\special{pa 1695 679}%
\special{pa 1700 665}%
\special{pa 1715 620}%
\special{pa 1725 588}%
\special{pa 1735 554}%
\special{pa 1740 536}%
\special{pa 1745 517}%
\special{pa 1760 457}%
\special{pa 1770 413}%
\special{pa 1780 365}%
\special{pa 1785 340}%
\special{pa 1790 314}%
\special{pa 1795 287}%
\special{pa 1810 200}%
\special{fp}%
\special{pn 13}%
\special{pn 13}%
\special{pa 400 1200}%
\special{pa 403 1187}%
\special{fp}%
\special{pa 417 1158}%
\special{pa 426 1148}%
\special{fp}%
\special{pa 449 1127}%
\special{pa 460 1118}%
\special{fp}%
\special{pa 486 1099}%
\special{pa 497 1092}%
\special{fp}%
\special{pa 524 1074}%
\special{pa 534 1066}%
\special{fp}%
\special{pa 562 1049}%
\special{pa 573 1041}%
\special{fp}%
\special{pa 600 1024}%
\special{pa 611 1017}%
\special{fp}%
\special{pa 637 998}%
\special{pa 648 990}%
\special{fp}%
\special{pa 674 971}%
\special{pa 684 963}%
\special{fp}%
\special{pa 709 943}%
\special{pa 719 934}%
\special{fp}%
\special{pa 743 912}%
\special{pa 753 903}%
\special{fp}%
\special{pa 776 880}%
\special{pa 785 870}%
\special{fp}%
\special{pa 806 846}%
\special{pa 815 836}%
\special{fp}%
\special{pa 835 811}%
\special{pa 842 800}%
\special{fp}%
\special{pa 861 773}%
\special{pa 868 762}%
\special{fp}%
\special{pa 884 734}%
\special{pa 891 723}%
\special{fp}%
\special{pa 906 694}%
\special{pa 912 682}%
\special{fp}%
\special{pa 925 653}%
\special{pa 931 641}%
\special{fp}%
\special{pa 943 611}%
\special{pa 948 598}%
\special{fp}%
\special{pa 959 568}%
\special{pa 963 555}%
\special{fp}%
\special{pa 973 525}%
\special{pa 977 512}%
\special{fp}%
\special{pa 986 481}%
\special{pa 989 468}%
\special{fp}%
\special{pa 997 437}%
\special{pa 1000 424}%
\special{fp}%
\special{pa 1008 392}%
\special{pa 1011 379}%
\special{fp}%
\special{pa 1017 348}%
\special{pa 1020 335}%
\special{fp}%
\special{pa 1026 303}%
\special{pa 1028 290}%
\special{fp}%
\special{pa 1034 258}%
\special{pa 1036 245}%
\special{fp}%
\special{pa 1041 213}%
\special{pa 1043 200}%
\special{fp}%
\special{pn 13}%
\special{pa 1357 2200}%
\special{pa 1359 2187}%
\special{fp}%
\special{pa 1364 2155}%
\special{pa 1366 2142}%
\special{fp}%
\special{pa 1372 2110}%
\special{pa 1374 2097}%
\special{fp}%
\special{pa 1380 2065}%
\special{pa 1383 2052}%
\special{fp}%
\special{pa 1389 2021}%
\special{pa 1392 2008}%
\special{fp}%
\special{pa 1400 1976}%
\special{pa 1403 1963}%
\special{fp}%
\special{pa 1411 1932}%
\special{pa 1414 1919}%
\special{fp}%
\special{pa 1423 1888}%
\special{pa 1427 1875}%
\special{fp}%
\special{pa 1437 1845}%
\special{pa 1441 1832}%
\special{fp}%
\special{pa 1452 1802}%
\special{pa 1457 1789}%
\special{fp}%
\special{pa 1469 1759}%
\special{pa 1475 1747}%
\special{fp}%
\special{pa 1488 1718}%
\special{pa 1494 1706}%
\special{fp}%
\special{pa 1509 1677}%
\special{pa 1516 1666}%
\special{fp}%
\special{pa 1532 1638}%
\special{pa 1539 1627}%
\special{fp}%
\special{pa 1558 1600}%
\special{pa 1565 1589}%
\special{fp}%
\special{pa 1585 1564}%
\special{pa 1594 1554}%
\special{fp}%
\special{pa 1615 1530}%
\special{pa 1624 1520}%
\special{fp}%
\special{pa 1647 1497}%
\special{pa 1657 1488}%
\special{fp}%
\special{pa 1681 1466}%
\special{pa 1691 1457}%
\special{fp}%
\special{pa 1716 1437}%
\special{pa 1726 1429}%
\special{fp}%
\special{pa 1752 1410}%
\special{pa 1763 1402}%
\special{fp}%
\special{pa 1789 1383}%
\special{pa 1800 1376}%
\special{fp}%
\special{pa 1827 1359}%
\special{pa 1838 1351}%
\special{fp}%
\special{pa 1866 1334}%
\special{pa 1876 1326}%
\special{fp}%
\special{pa 1903 1308}%
\special{pa 1914 1301}%
\special{fp}%
\special{pa 1940 1282}%
\special{pa 1951 1273}%
\special{fp}%
\special{pa 1974 1252}%
\special{pa 1983 1242}%
\special{fp}%
\special{pa 1997 1213}%
\special{pa 2000 1200}%
\special{fp}%
\put(15.5000,-3.0000){\makebox(0,0)[lb]{{\large $\psi$}}}%
%
\special{pn 8}%
\special{pn 8}%
\special{pa 400 200}%
\special{pa 400 263}%
\special{fp}%
\special{pa 400 293}%
\special{pa 400 302}%
\special{fp}%
\special{pa 400 332}%
\special{pa 400 395}%
\special{fp}%
\special{pa 400 426}%
\special{pa 400 434}%
\special{fp}%
\special{pa 400 464}%
\special{pa 400 527}%
\special{fp}%
\special{pa 400 558}%
\special{pa 400 566}%
\special{fp}%
\special{pa 400 596}%
\special{pa 400 659}%
\special{fp}%
\special{pa 400 690}%
\special{pa 400 698}%
\special{fp}%
\special{pa 400 728}%
\special{pa 400 791}%
\special{fp}%
\special{pa 400 821}%
\special{pa 400 829}%
\special{fp}%
\special{pa 400 860}%
\special{pa 400 923}%
\special{fp}%
\special{pa 400 953}%
\special{pa 400 961}%
\special{fp}%
\special{pa 400 992}%
\special{pa 400 1055}%
\special{fp}%
\special{pa 400 1085}%
\special{pa 400 1093}%
\special{fp}%
\special{pa 400 1124}%
\special{pa 400 1187}%
\special{fp}%
\special{pa 400 1217}%
\special{pa 400 1225}%
\special{fp}%
\special{pa 400 1256}%
\special{pa 400 1319}%
\special{fp}%
\special{pa 400 1349}%
\special{pa 400 1357}%
\special{fp}%
\special{pa 400 1388}%
\special{pa 400 1451}%
\special{fp}%
\special{pa 400 1481}%
\special{pa 400 1489}%
\special{fp}%
\special{pa 400 1520}%
\special{pa 400 1583}%
\special{fp}%
\special{pa 400 1613}%
\special{pa 400 1621}%
\special{fp}%
\special{pa 400 1652}%
\special{pa 400 1715}%
\special{fp}%
\special{pa 400 1745}%
\special{pa 400 1753}%
\special{fp}%
\special{pa 400 1784}%
\special{pa 400 1847}%
\special{fp}%
\special{pa 400 1877}%
\special{pa 400 1885}%
\special{fp}%
\special{pa 400 1916}%
\special{pa 400 1979}%
\special{fp}%
\special{pa 400 2009}%
\special{pa 400 2017}%
\special{fp}%
\special{pa 400 2048}%
\special{pa 400 2111}%
\special{fp}%
\special{pa 400 2141}%
\special{pa 400 2149}%
\special{fp}%
\special{pa 400 2180}%
\special{pa 400 2200}%
\special{fp}%
\special{pn 8}%
\special{pa 2000 2200}%
\special{pa 2000 2137}%
\special{fp}%
\special{pa 2000 2107}%
\special{pa 2000 2098}%
\special{fp}%
\special{pa 2000 2068}%
\special{pa 2000 2005}%
\special{fp}%
\special{pa 2000 1974}%
\special{pa 2000 1966}%
\special{fp}%
\special{pa 2000 1936}%
\special{pa 2000 1873}%
\special{fp}%
\special{pa 2000 1842}%
\special{pa 2000 1834}%
\special{fp}%
\special{pa 2000 1804}%
\special{pa 2000 1741}%
\special{fp}%
\special{pa 2000 1710}%
\special{pa 2000 1702}%
\special{fp}%
\special{pa 2000 1672}%
\special{pa 2000 1609}%
\special{fp}%
\special{pa 2000 1579}%
\special{pa 2000 1571}%
\special{fp}%
\special{pa 2000 1540}%
\special{pa 2000 1477}%
\special{fp}%
\special{pa 2000 1447}%
\special{pa 2000 1439}%
\special{fp}%
\special{pa 2000 1408}%
\special{pa 2000 1345}%
\special{fp}%
\special{pa 2000 1315}%
\special{pa 2000 1307}%
\special{fp}%
\special{pa 2000 1276}%
\special{pa 2000 1213}%
\special{fp}%
\special{pa 2000 1183}%
\special{pa 2000 1175}%
\special{fp}%
\special{pa 2000 1144}%
\special{pa 2000 1081}%
\special{fp}%
\special{pa 2000 1051}%
\special{pa 2000 1043}%
\special{fp}%
\special{pa 2000 1012}%
\special{pa 2000 949}%
\special{fp}%
\special{pa 2000 919}%
\special{pa 2000 911}%
\special{fp}%
\special{pa 2000 880}%
\special{pa 2000 817}%
\special{fp}%
\special{pa 2000 787}%
\special{pa 2000 779}%
\special{fp}%
\special{pa 2000 748}%
\special{pa 2000 685}%
\special{fp}%
\special{pa 2000 655}%
\special{pa 2000 647}%
\special{fp}%
\special{pa 2000 616}%
\special{pa 2000 553}%
\special{fp}%
\special{pa 2000 523}%
\special{pa 2000 515}%
\special{fp}%
\special{pa 2000 484}%
\special{pa 2000 421}%
\special{fp}%
\special{pa 2000 391}%
\special{pa 2000 383}%
\special{fp}%
\special{pa 2000 352}%
\special{pa 2000 289}%
\special{fp}%
\special{pa 2000 259}%
\special{pa 2000 251}%
\special{fp}%
\special{pa 2000 220}%
\special{pa 2000 200}%
\special{fp}%
\put(22.3000,-12.5000){\makebox(0,0)[lb]{$r$}}%
\put(20.3000,-11.8000){\makebox(0,0)[lb]{$1$}}%
\put(3.6000,-11.7000){\makebox(0,0)[rb]{$-1$}}%
\put(14.8000,-13.0000){\makebox(0,0)[lb]{{\tiny $r_0$}}}%
\put(9.4000,-10.4000){\makebox(0,0)[lb]{{\tiny $\psi(r_0)$}}}%
%
\special{pn 4}%
\special{pa 1520 980}%
\special{pa 1200 980}%
\special{dt 0.045}%
\special{pa 1520 980}%
\special{pa 1520 1200}%
\special{dt 0.045}%
%
\special{pn 8}%
\special{pa 1840 200}%
\special{pa 1840 1200}%
\special{da 0.070}%
\put(17.8000,-13.0000){\makebox(0,0)[lb]{{\tiny $r_1$}}}%
%
\special{pn 8}%
\special{pa 200 1200}%
\special{pa 2200 1200}%
\special{fp}%
\special{sh 1}%
\special{pa 2200 1200}%
\special{pa 2133 1180}%
\special{pa 2147 1200}%
\special{pa 2133 1220}%
\special{pa 2200 1200}%
\special{fp}%
%
\special{pn 8}%
\special{pa 1200 2200}%
\special{pa 1200 200}%
\special{fp}%
\put(10.5000,-1.8000){\makebox(0,0)[lb]{$r=R$}}%
\end{picture}}}
\caption{The behavior of the graph of $\psi$ in Lemma \ref{thm:psi-shape1}.}
\end{figure}

\begin{lemma}\label{thm:psi-shape2}
If there exists $r_0\in\left(-1,R\right)$ with $0<\psi(r_0)<\eta(r_0)$, there exists $C\in\left(\psi(r_0),+\infty\right)$ such that
\begin{equation*}
\lim_{r\uparrow R}\psi(r)=C.
\end{equation*}
\end{lemma}

\begin{proof}
First, we consider the case $k=1$.
For all $r\in(r_0,0)$, we find $\psi'(r)>0$ and $0<\psi(r)<\eta(r)$.
Also, we have
\begin{align*}
\psi'(r)&=\frac{1}{1-r^2}\Bigl(\psi(r)^2+1\Bigl)\Bigl((n-1)r\psi(r)+\sqrt{1-r^2}\Bigl)\\
&<\frac{1}{\sqrt{1-r^2}}\left(\psi(r)^2+1\right).
\end{align*} 
Therefore, we find
\begin{equation*}
\frac{\psi'(r)}{\psi(r)^2+1}<\frac{1}{\sqrt{1-r^2}}.
\end{equation*}
By integrating from $r_0$ to $r$, we have
\begin{equation*}
\arctan{\psi(r)}<\arcsin{r}-\arcsin{r_0}+\arctan{\psi(r_0)}=:h_2(r).
\end{equation*}
Here, $h_2$ is increasing on $(r_0,0)$ and 
\begin{equation*}
h_2(r_0)=\arctan{\psi(r_0)},\quad h_2(0)=\arctan{\psi(r_0)}-\arcsin{r_0}.
\end{equation*}
Since we find
\begin{align*}
\psi(r_0)<\eta(r_0)=-\frac{\sqrt{1-r_0^2}}{(n-1)r_0}\le-\frac{\sqrt{1-r_0^2}}{r_0}=\tan{\left(\arcsin{r_0}+\frac{\pi}{2}\right)},
\end{align*}
we have
\begin{equation*}
h_2(0)=\arctan{\psi(r_0)}-\arcsin{r_0}<\frac{\pi}{2}.
\end{equation*}
Therefore, $\tan{\left(h_2(r)\right)}$ is defined on $(r_0,0]$ and
\begin{equation*}
\psi(r)<\tan{\left(h_2(r)\right)}.
\end{equation*}
Then, we obtain the statemant of this lemma for $k=1$.

Next, we consider the case $k=2$, $3$, $4$ or $6$.
For all $r\in(r_0,R)$, we find $\psi(r)>0$ and $0<\psi(r)<\eta(r)$.
Also, we have
\begin{align*}
\psi'(r)&=\frac{1}{k(1-r^2)}\Bigl(\psi(r)^2+1\Bigl)\Bigl((n-1)(r-R)\psi(r)+\sqrt{1-r^2}\Bigl)\\
&<\frac{1}{2\sqrt{1-r^2}}\left(\psi(r)^2+1\right).
\end{align*} 
Therefore, we find
\begin{equation*}
\frac{\psi'(r)}{\psi(r)^2+1}<\frac{1}{2\sqrt{1-r^2}}.
\end{equation*}
By integrating from $r_0$ to $r$, we have
\begin{equation*}
\arctan{\psi(r)}<\frac{1}{2}\arcsin{r}-\frac{1}{2}\arcsin{r_0}+\arctan{\psi(r_0)}=:\hat{h}_2(r).
\end{equation*}
Here, $\hat{h}_2$ is increasing on $(r_0,R)$ and 
\begin{equation*}
\hat{h}_2(r_0)=\arctan{\psi(r_0)},\quad \hat{h}_2(R)=\arctan{\psi(r_0)}+\frac{1}{2}\arcsin{R}-\frac{1}{2}\arcsin{r_0}.
\end{equation*}
Since we find
\begin{align*}
\psi(r_0)<\eta(r_0)&=-\frac{\sqrt{1-r_0^2}}{(n-1)(r_0-R)}\\
&<-\frac{\sqrt{1-r_0^2}+\sqrt{1-R^2}}{r_0-R}\\
&=\tan{\left(\frac{1}{2}\arcsin{r_0}-\frac{1}{2}\arcsin{R}+\frac{\pi}{2}\right)},
\end{align*}
we have
\begin{equation*}
\hat{h}_2(R)=\arctan{\psi(r_0)}-\arcsin{r_0}<\frac{\pi}{2}.
\end{equation*}
Therefore, $\tan{\left(\hat{h}_2(r)\right)}$ is defined on $(r_0,R]$ and
\begin{equation*}
\psi(r)<\tan{\left(\hat{h}_2(r)\right)}.
\end{equation*}
Then, we obtain the statement of this lemma.
\end{proof}

\begin{figure}[H]
\centering
\scalebox{1.0}{{\unitlength 0.1in%
\begin{picture}(20.9000,21.5000)(1.4000,-22.0000)%
\special{pn 20}%
\special{pn 20}%
\special{pa 542 2200}%
\special{pa 544 2180}%
\special{ip}%
\special{pa 546 2154}%
\special{pa 548 2134}%
\special{ip}%
\special{pa 551 2109}%
\special{pa 553 2088}%
\special{ip}%
\special{pa 556 2063}%
\special{pa 559 2043}%
\special{ip}%
\special{pa 562 2018}%
\special{pa 564 1997}%
\special{ip}%
\special{pa 568 1972}%
\special{pa 570 1954}%
\special{pa 570 1952}%
\special{ip}%
\special{pa 574 1927}%
\special{pa 575 1918}%
\special{pa 577 1906}%
\special{ip}%
\special{pa 580 1881}%
\special{pa 584 1861}%
\special{ip}%
\special{pa 588 1836}%
\special{pa 590 1822}%
\special{pa 591 1816}%
\special{ip}%
\special{pa 595 1791}%
\special{pa 599 1771}%
\special{ip}%
\special{pa 604 1746}%
\special{pa 605 1739}%
\special{pa 608 1726}%
\special{ip}%
\special{pa 613 1701}%
\special{pa 617 1681}%
\special{ip}%
\special{pa 623 1656}%
\special{pa 625 1646}%
\special{pa 627 1636}%
\special{ip}%
\special{pa 633 1611}%
\special{pa 635 1604}%
\special{pa 638 1592}%
\special{ip}%
\special{pa 645 1567}%
\special{pa 645 1566}%
\special{pa 650 1548}%
\special{pa 650 1547}%
\special{ip}%
\special{pa 657 1523}%
\special{pa 660 1514}%
\special{pa 663 1503}%
\special{ip}%
\special{pa 671 1479}%
\special{pa 675 1467}%
\special{pa 678 1460}%
\special{ip}%
\special{pa 686 1436}%
\special{pa 690 1425}%
\special{pa 693 1417}%
\special{ip}%
\special{pa 702 1393}%
\special{pa 705 1386}%
\special{pa 710 1374}%
\special{ip}%
\special{pa 720 1351}%
\special{pa 725 1339}%
\special{pa 728 1332}%
\special{ip}%
\special{pa 739 1309}%
\special{pa 740 1307}%
\special{pa 748 1291}%
\special{ip}%
\special{pa 760 1268}%
\special{pa 770 1250}%
\special{ip}%
\special{pa 783 1228}%
\special{pa 785 1224}%
\special{pa 793 1211}%
\special{ip}%
\special{pa 807 1189}%
\special{pa 810 1185}%
\special{pa 815 1177}%
\special{pa 818 1172}%
\special{ip}%
\special{pa 833 1151}%
\special{pa 835 1148}%
\special{pa 840 1142}%
\special{pa 845 1135}%
\special{ip}%
\special{pa 860 1115}%
\special{pa 873 1099}%
\special{ip}%
\special{pa 890 1079}%
\special{pa 900 1067}%
\special{pa 903 1064}%
\special{ip}%
\special{pa 920 1045}%
\special{pa 925 1040}%
\special{pa 930 1034}%
\special{pa 980 984}%
\special{pa 985 980}%
\special{pa 995 970}%
\special{pa 1000 966}%
\special{pa 1005 961}%
\special{pa 1010 957}%
\special{pa 1015 952}%
\special{pa 1020 948}%
\special{pa 1025 943}%
\special{pa 1030 939}%
\special{pa 1035 934}%
\special{pa 1045 926}%
\special{pa 1050 921}%
\special{pa 1065 909}%
\special{pa 1070 904}%
\special{pa 1100 880}%
\special{pa 1105 875}%
\special{pa 1160 831}%
\special{pa 1165 828}%
\special{pa 1205 796}%
\special{fp}%
\special{pn 20}%
\special{pa 1221 783}%
\special{pa 1235 772}%
\special{pa 1240 769}%
\special{pa 1242 768}%
\special{ip}%
\special{pa 1258 755}%
\special{pa 1278 739}%
\special{ip}%
\special{pa 1294 726}%
\special{pa 1295 725}%
\special{pa 1300 720}%
\special{pa 1313 709}%
\special{ip}%
\special{pa 1329 696}%
\special{pa 1330 696}%
\special{pa 1335 691}%
\special{pa 1349 680}%
\special{ip}%
\special{pa 1365 666}%
\special{pa 1365 666}%
\special{pa 1370 661}%
\special{pa 1375 657}%
\special{pa 1380 652}%
\special{pa 1384 649}%
\special{ip}%
\special{pa 1399 635}%
\special{pa 1400 634}%
\special{pa 1405 630}%
\special{pa 1415 620}%
\special{pa 1418 618}%
\special{ip}%
\special{pa 1432 604}%
\special{pa 1451 585}%
\special{ip}%
\special{pa 1465 571}%
\special{pa 1470 566}%
\special{pa 1475 560}%
\special{pa 1480 555}%
\special{pa 1483 552}%
\special{ip}%
\special{pa 1496 537}%
\special{pa 1500 533}%
\special{pa 1513 517}%
\special{ip}%
\special{pa 1526 501}%
\special{pa 1540 485}%
\special{pa 1543 481}%
\special{ip}%
\special{pa 1555 465}%
\special{pa 1560 458}%
\special{pa 1565 452}%
\special{pa 1570 445}%
\special{pa 1570 444}%
\special{ip}%
\special{pa 1582 427}%
\special{pa 1585 423}%
\special{pa 1590 415}%
\special{pa 1595 408}%
\special{pa 1596 406}%
\special{ip}%
\special{pa 1607 389}%
\special{pa 1615 376}%
\special{pa 1620 367}%
\special{pa 1620 366}%
\special{ip}%
\special{pa 1631 349}%
\special{pa 1640 332}%
\special{pa 1643 326}%
\special{ip}%
\special{pa 1653 308}%
\special{pa 1660 293}%
\special{pa 1664 285}%
\special{ip}%
\special{pa 1673 266}%
\special{pa 1675 261}%
\special{pa 1680 249}%
\special{pa 1683 243}%
\special{ip}%
\special{pa 1691 224}%
\special{pa 1695 214}%
\special{pa 1700 201}%
\special{pa 1701 200}%
\special{ip}%
\special{pn 20}%
\special{pa 2144 2179}%
\special{pa 2145 2167}%
\special{pa 2147 2153}%
\special{ip}%
\special{pa 2149 2132}%
\special{pa 2150 2119}%
\special{pa 2152 2105}%
\special{ip}%
\special{pa 2154 2084}%
\special{pa 2155 2073}%
\special{pa 2157 2058}%
\special{ip}%
\special{pa 2159 2037}%
\special{pa 2160 2031}%
\special{pa 2163 2011}%
\special{ip}%
\special{pa 2165 1990}%
\special{pa 2169 1963}%
\special{ip}%
\special{pa 2172 1942}%
\special{pa 2175 1918}%
\special{pa 2175 1916}%
\special{ip}%
\special{pa 2178 1895}%
\special{pa 2180 1884}%
\special{pa 2182 1869}%
\special{ip}%
\special{pa 2186 1848}%
\special{pa 2190 1822}%
\special{ip}%
\special{pn 13}%
\special{pn 13}%
\special{pa 400 1200}%
\special{pa 403 1187}%
\special{fp}%
\special{pa 417 1158}%
\special{pa 426 1148}%
\special{fp}%
\special{pa 449 1127}%
\special{pa 460 1118}%
\special{fp}%
\special{pa 486 1099}%
\special{pa 497 1092}%
\special{fp}%
\special{pa 524 1074}%
\special{pa 534 1066}%
\special{fp}%
\special{pa 562 1049}%
\special{pa 573 1041}%
\special{fp}%
\special{pa 600 1024}%
\special{pa 611 1017}%
\special{fp}%
\special{pa 637 998}%
\special{pa 648 990}%
\special{fp}%
\special{pa 674 971}%
\special{pa 684 963}%
\special{fp}%
\special{pa 709 943}%
\special{pa 719 934}%
\special{fp}%
\special{pa 743 912}%
\special{pa 753 903}%
\special{fp}%
\special{pa 776 880}%
\special{pa 785 870}%
\special{fp}%
\special{pa 806 846}%
\special{pa 815 836}%
\special{fp}%
\special{pa 835 811}%
\special{pa 842 800}%
\special{fp}%
\special{pa 861 773}%
\special{pa 868 762}%
\special{fp}%
\special{pa 884 734}%
\special{pa 891 723}%
\special{fp}%
\special{pa 906 694}%
\special{pa 912 682}%
\special{fp}%
\special{pa 925 653}%
\special{pa 931 641}%
\special{fp}%
\special{pa 943 611}%
\special{pa 948 598}%
\special{fp}%
\special{pa 959 568}%
\special{pa 963 555}%
\special{fp}%
\special{pa 973 525}%
\special{pa 977 512}%
\special{fp}%
\special{pa 986 481}%
\special{pa 989 468}%
\special{fp}%
\special{pa 997 437}%
\special{pa 1000 424}%
\special{fp}%
\special{pa 1008 392}%
\special{pa 1011 379}%
\special{fp}%
\special{pa 1017 348}%
\special{pa 1020 335}%
\special{fp}%
\special{pa 1026 303}%
\special{pa 1028 290}%
\special{fp}%
\special{pa 1034 258}%
\special{pa 1036 245}%
\special{fp}%
\special{pa 1041 213}%
\special{pa 1043 200}%
\special{fp}%
\special{pn 13}%
\special{pa 1357 2200}%
\special{pa 1359 2187}%
\special{fp}%
\special{pa 1364 2155}%
\special{pa 1366 2142}%
\special{fp}%
\special{pa 1372 2110}%
\special{pa 1374 2097}%
\special{fp}%
\special{pa 1380 2065}%
\special{pa 1383 2052}%
\special{fp}%
\special{pa 1389 2021}%
\special{pa 1392 2008}%
\special{fp}%
\special{pa 1400 1976}%
\special{pa 1403 1963}%
\special{fp}%
\special{pa 1411 1932}%
\special{pa 1414 1919}%
\special{fp}%
\special{pa 1423 1888}%
\special{pa 1427 1875}%
\special{fp}%
\special{pa 1437 1845}%
\special{pa 1441 1832}%
\special{fp}%
\special{pa 1452 1802}%
\special{pa 1457 1789}%
\special{fp}%
\special{pa 1469 1759}%
\special{pa 1475 1747}%
\special{fp}%
\special{pa 1488 1718}%
\special{pa 1494 1706}%
\special{fp}%
\special{pa 1509 1677}%
\special{pa 1516 1666}%
\special{fp}%
\special{pa 1532 1638}%
\special{pa 1539 1627}%
\special{fp}%
\special{pa 1558 1600}%
\special{pa 1565 1589}%
\special{fp}%
\special{pa 1585 1564}%
\special{pa 1594 1554}%
\special{fp}%
\special{pa 1615 1530}%
\special{pa 1624 1520}%
\special{fp}%
\special{pa 1647 1497}%
\special{pa 1657 1488}%
\special{fp}%
\special{pa 1681 1466}%
\special{pa 1691 1457}%
\special{fp}%
\special{pa 1716 1437}%
\special{pa 1726 1429}%
\special{fp}%
\special{pa 1752 1410}%
\special{pa 1763 1402}%
\special{fp}%
\special{pa 1789 1383}%
\special{pa 1800 1376}%
\special{fp}%
\special{pa 1827 1359}%
\special{pa 1838 1351}%
\special{fp}%
\special{pa 1866 1334}%
\special{pa 1876 1326}%
\special{fp}%
\special{pa 1903 1308}%
\special{pa 1914 1301}%
\special{fp}%
\special{pa 1940 1282}%
\special{pa 1951 1273}%
\special{fp}%
\special{pa 1974 1252}%
\special{pa 1983 1242}%
\special{fp}%
\special{pa 1997 1213}%
\special{pa 2000 1200}%
\special{fp}%
\put(13.2000,-7.0000){\makebox(0,0)[lb]{{\large $\psi$}}}%
%
\special{pn 8}%
\special{pn 8}%
\special{pa 400 200}%
\special{pa 400 263}%
\special{fp}%
\special{pa 400 293}%
\special{pa 400 302}%
\special{fp}%
\special{pa 400 332}%
\special{pa 400 395}%
\special{fp}%
\special{pa 400 426}%
\special{pa 400 434}%
\special{fp}%
\special{pa 400 464}%
\special{pa 400 527}%
\special{fp}%
\special{pa 400 558}%
\special{pa 400 566}%
\special{fp}%
\special{pa 400 596}%
\special{pa 400 659}%
\special{fp}%
\special{pa 400 690}%
\special{pa 400 698}%
\special{fp}%
\special{pa 400 728}%
\special{pa 400 791}%
\special{fp}%
\special{pa 400 821}%
\special{pa 400 829}%
\special{fp}%
\special{pa 400 860}%
\special{pa 400 923}%
\special{fp}%
\special{pa 400 953}%
\special{pa 400 961}%
\special{fp}%
\special{pa 400 992}%
\special{pa 400 1055}%
\special{fp}%
\special{pa 400 1085}%
\special{pa 400 1093}%
\special{fp}%
\special{pa 400 1124}%
\special{pa 400 1187}%
\special{fp}%
\special{pa 400 1217}%
\special{pa 400 1225}%
\special{fp}%
\special{pa 400 1256}%
\special{pa 400 1319}%
\special{fp}%
\special{pa 400 1349}%
\special{pa 400 1357}%
\special{fp}%
\special{pa 400 1388}%
\special{pa 400 1451}%
\special{fp}%
\special{pa 400 1481}%
\special{pa 400 1489}%
\special{fp}%
\special{pa 400 1520}%
\special{pa 400 1583}%
\special{fp}%
\special{pa 400 1613}%
\special{pa 400 1621}%
\special{fp}%
\special{pa 400 1652}%
\special{pa 400 1715}%
\special{fp}%
\special{pa 400 1745}%
\special{pa 400 1753}%
\special{fp}%
\special{pa 400 1784}%
\special{pa 400 1847}%
\special{fp}%
\special{pa 400 1877}%
\special{pa 400 1885}%
\special{fp}%
\special{pa 400 1916}%
\special{pa 400 1979}%
\special{fp}%
\special{pa 400 2009}%
\special{pa 400 2017}%
\special{fp}%
\special{pa 400 2048}%
\special{pa 400 2111}%
\special{fp}%
\special{pa 400 2141}%
\special{pa 400 2149}%
\special{fp}%
\special{pa 400 2180}%
\special{pa 400 2200}%
\special{fp}%
\special{pn 8}%
\special{pa 2000 2200}%
\special{pa 2000 2137}%
\special{fp}%
\special{pa 2000 2107}%
\special{pa 2000 2098}%
\special{fp}%
\special{pa 2000 2068}%
\special{pa 2000 2005}%
\special{fp}%
\special{pa 2000 1974}%
\special{pa 2000 1966}%
\special{fp}%
\special{pa 2000 1936}%
\special{pa 2000 1873}%
\special{fp}%
\special{pa 2000 1842}%
\special{pa 2000 1834}%
\special{fp}%
\special{pa 2000 1804}%
\special{pa 2000 1741}%
\special{fp}%
\special{pa 2000 1710}%
\special{pa 2000 1702}%
\special{fp}%
\special{pa 2000 1672}%
\special{pa 2000 1609}%
\special{fp}%
\special{pa 2000 1579}%
\special{pa 2000 1571}%
\special{fp}%
\special{pa 2000 1540}%
\special{pa 2000 1477}%
\special{fp}%
\special{pa 2000 1447}%
\special{pa 2000 1439}%
\special{fp}%
\special{pa 2000 1408}%
\special{pa 2000 1345}%
\special{fp}%
\special{pa 2000 1315}%
\special{pa 2000 1307}%
\special{fp}%
\special{pa 2000 1276}%
\special{pa 2000 1213}%
\special{fp}%
\special{pa 2000 1183}%
\special{pa 2000 1175}%
\special{fp}%
\special{pa 2000 1144}%
\special{pa 2000 1081}%
\special{fp}%
\special{pa 2000 1051}%
\special{pa 2000 1043}%
\special{fp}%
\special{pa 2000 1012}%
\special{pa 2000 949}%
\special{fp}%
\special{pa 2000 919}%
\special{pa 2000 911}%
\special{fp}%
\special{pa 2000 880}%
\special{pa 2000 817}%
\special{fp}%
\special{pa 2000 787}%
\special{pa 2000 779}%
\special{fp}%
\special{pa 2000 748}%
\special{pa 2000 685}%
\special{fp}%
\special{pa 2000 655}%
\special{pa 2000 647}%
\special{fp}%
\special{pa 2000 616}%
\special{pa 2000 553}%
\special{fp}%
\special{pa 2000 523}%
\special{pa 2000 515}%
\special{fp}%
\special{pa 2000 484}%
\special{pa 2000 421}%
\special{fp}%
\special{pa 2000 391}%
\special{pa 2000 383}%
\special{fp}%
\special{pa 2000 352}%
\special{pa 2000 289}%
\special{fp}%
\special{pa 2000 259}%
\special{pa 2000 251}%
\special{fp}%
\special{pa 2000 220}%
\special{pa 2000 200}%
\special{fp}%
\put(22.3000,-12.5000){\makebox(0,0)[lb]{$r$}}%
\put(20.3000,-11.8000){\makebox(0,0)[lb]{$1$}}%
\put(3.6000,-11.7000){\makebox(0,0)[rb]{$-1$}}%
%
\special{pn 4}%
\special{pa 930 1030}%
\special{pa 930 1200}%
\special{dt 0.045}%
\special{pa 930 1030}%
\special{pa 1200 1030}%
\special{dt 0.045}%
\put(9.0000,-13.0000){\makebox(0,0)[lb]{{\tiny $r_0$}}}%
\put(12.2000,-10.6000){\makebox(0,0)[lb]{{\tiny $\psi(r_0)$}}}%
%
\special{pn 8}%
\special{pa 200 1200}%
\special{pa 2200 1200}%
\special{fp}%
\special{sh 1}%
\special{pa 2200 1200}%
\special{pa 2133 1180}%
\special{pa 2147 1200}%
\special{pa 2133 1220}%
\special{pa 2200 1200}%
\special{fp}%
%
\special{pn 8}%
\special{pa 1200 2200}%
\special{pa 1200 200}%
\special{fp}%
\put(10.5000,-1.8000){\makebox(0,0)[lb]{$r=R$}}%
\end{picture}}}
\caption{The behavior of the graph of $\psi$ in Lemma \ref{thm:psi-shape2}}
\end{figure}

\begin{lemma}\label{thm:psi-shape3}
If there exists $r_0\in(-1,R)$ with $\psi(r_0)>\eta(r_0)$, there exists $r_1\in(-1,r_0)$ such that 
\begin{equation*}
\lim_{r\downarrow r_1}\psi(r)=+\infty.
\end{equation*}
\end{lemma}

\begin{proof}
For all $r\in(-1,r_0)$, we find $\psi'(r)<0$ and $\psi(r)>\eta(r)$.
Also, we have
\begin{align*}
\psi'(r)&=\frac{1}{k(1-r^2)}\left(\psi(r)^2+1\right)\left((n-1)(r-R)\psi(r)+\sqrt{1-r^2}\right)\\
&<\frac{1}{k(1-r^2)}\left((n-1)\psi(r_0)(r-R)+\sqrt{1-r^2}\right)\psi(r)^2.
\end{align*}
Therefore, we find
\begin{equation*}
\frac{\psi'(r)}{\psi(r)^2}<\frac{(n-1)}{k}\psi(r_0)\frac{r-R}{1-r^2}+\frac{1}{k\sqrt{1-r^2}}.
\end{equation*}
By integrating from $r_0$ to $r$, we have
\begin{align*}
\frac{1}{\psi(r)}<~&\frac{(n-1)\psi(r_0)}{2k}\log{(1-r^2)}-\frac{1}{k}\arcsin{r}\\
&+\frac{(n-1)R\psi(r_0)}{2k}\log{\frac{1+r}{1-r}}-\frac{(n-1)\psi(r_0)}{2k}\log{(1-r_0^2)}\\
&+\frac{1}{k}\arcsin{r_0}-\frac{(n-1)R\psi(r_0)}{2k}\log{\frac{1+r_0}{1-r_0}}+\frac{1}{\psi(r_0)}=:h_3(r).
\end{align*}
Here, $h_3$ is increasing on $(-1,r_0)$ and 
\begin{equation*}
h_3(r_0)=\frac{1}{\psi(r_0)}>0,\quad \lim_{r\downarrow -1}h_3(r)=-\infty.
\end{equation*}
Therefore, there exists $\overline{r}_1\in(-1,r_0)$ with $h_3(\overline{r}_1)=0$ and
\begin{equation*}
\psi(r)>\frac{1}{h_3(r)}\rightarrow+\infty\quad(~r~\downarrow~\overline{r}_1~)
\end{equation*}
Then, we obtain the statemant of this lemma.
\end{proof}

\begin{figure}[H]
\centering
\scalebox{1.0}{{\unitlength 0.1in%
\begin{picture}(20.9000,21.5000)(1.5000,-22.0000)%
\special{pn 13}%
\special{pn 13}%
\special{pa 400 1200}%
\special{pa 403 1187}%
\special{fp}%
\special{pa 417 1158}%
\special{pa 426 1148}%
\special{fp}%
\special{pa 449 1127}%
\special{pa 460 1118}%
\special{fp}%
\special{pa 486 1099}%
\special{pa 497 1092}%
\special{fp}%
\special{pa 524 1074}%
\special{pa 534 1066}%
\special{fp}%
\special{pa 562 1049}%
\special{pa 573 1041}%
\special{fp}%
\special{pa 600 1024}%
\special{pa 611 1017}%
\special{fp}%
\special{pa 637 998}%
\special{pa 648 990}%
\special{fp}%
\special{pa 674 971}%
\special{pa 684 963}%
\special{fp}%
\special{pa 709 943}%
\special{pa 719 934}%
\special{fp}%
\special{pa 743 912}%
\special{pa 753 903}%
\special{fp}%
\special{pa 776 880}%
\special{pa 785 870}%
\special{fp}%
\special{pa 806 846}%
\special{pa 815 836}%
\special{fp}%
\special{pa 835 811}%
\special{pa 842 800}%
\special{fp}%
\special{pa 861 773}%
\special{pa 868 762}%
\special{fp}%
\special{pa 884 734}%
\special{pa 891 723}%
\special{fp}%
\special{pa 906 694}%
\special{pa 912 682}%
\special{fp}%
\special{pa 925 653}%
\special{pa 931 641}%
\special{fp}%
\special{pa 943 611}%
\special{pa 948 598}%
\special{fp}%
\special{pa 959 568}%
\special{pa 963 555}%
\special{fp}%
\special{pa 973 525}%
\special{pa 977 512}%
\special{fp}%
\special{pa 986 481}%
\special{pa 989 468}%
\special{fp}%
\special{pa 997 437}%
\special{pa 1000 424}%
\special{fp}%
\special{pa 1008 392}%
\special{pa 1011 379}%
\special{fp}%
\special{pa 1017 348}%
\special{pa 1020 335}%
\special{fp}%
\special{pa 1026 303}%
\special{pa 1028 290}%
\special{fp}%
\special{pa 1034 258}%
\special{pa 1036 245}%
\special{fp}%
\special{pa 1041 213}%
\special{pa 1043 200}%
\special{fp}%
\special{pn 13}%
\special{pa 1357 2200}%
\special{pa 1359 2187}%
\special{fp}%
\special{pa 1364 2155}%
\special{pa 1366 2142}%
\special{fp}%
\special{pa 1372 2110}%
\special{pa 1374 2097}%
\special{fp}%
\special{pa 1380 2065}%
\special{pa 1383 2052}%
\special{fp}%
\special{pa 1389 2021}%
\special{pa 1392 2008}%
\special{fp}%
\special{pa 1400 1976}%
\special{pa 1403 1963}%
\special{fp}%
\special{pa 1411 1932}%
\special{pa 1414 1919}%
\special{fp}%
\special{pa 1423 1888}%
\special{pa 1427 1875}%
\special{fp}%
\special{pa 1437 1845}%
\special{pa 1441 1832}%
\special{fp}%
\special{pa 1452 1802}%
\special{pa 1457 1789}%
\special{fp}%
\special{pa 1469 1759}%
\special{pa 1475 1747}%
\special{fp}%
\special{pa 1488 1718}%
\special{pa 1494 1706}%
\special{fp}%
\special{pa 1509 1677}%
\special{pa 1516 1666}%
\special{fp}%
\special{pa 1532 1638}%
\special{pa 1539 1627}%
\special{fp}%
\special{pa 1558 1600}%
\special{pa 1565 1589}%
\special{fp}%
\special{pa 1585 1564}%
\special{pa 1594 1554}%
\special{fp}%
\special{pa 1615 1530}%
\special{pa 1624 1520}%
\special{fp}%
\special{pa 1647 1497}%
\special{pa 1657 1488}%
\special{fp}%
\special{pa 1681 1466}%
\special{pa 1691 1457}%
\special{fp}%
\special{pa 1716 1437}%
\special{pa 1726 1429}%
\special{fp}%
\special{pa 1752 1410}%
\special{pa 1763 1402}%
\special{fp}%
\special{pa 1789 1383}%
\special{pa 1800 1376}%
\special{fp}%
\special{pa 1827 1359}%
\special{pa 1838 1351}%
\special{fp}%
\special{pa 1866 1334}%
\special{pa 1876 1326}%
\special{fp}%
\special{pa 1903 1308}%
\special{pa 1914 1301}%
\special{fp}%
\special{pa 1940 1282}%
\special{pa 1951 1273}%
\special{fp}%
\special{pa 1974 1252}%
\special{pa 1983 1242}%
\special{fp}%
\special{pa 1997 1213}%
\special{pa 2000 1200}%
\special{fp}%
%
\special{pn 8}%
\special{pn 8}%
\special{pa 400 2200}%
\special{pa 400 2137}%
\special{fp}%
\special{pa 400 2107}%
\special{pa 400 2098}%
\special{fp}%
\special{pa 400 2068}%
\special{pa 400 2005}%
\special{fp}%
\special{pa 400 1974}%
\special{pa 400 1966}%
\special{fp}%
\special{pa 400 1936}%
\special{pa 400 1873}%
\special{fp}%
\special{pa 400 1842}%
\special{pa 400 1834}%
\special{fp}%
\special{pa 400 1804}%
\special{pa 400 1741}%
\special{fp}%
\special{pa 400 1710}%
\special{pa 400 1702}%
\special{fp}%
\special{pa 400 1672}%
\special{pa 400 1609}%
\special{fp}%
\special{pa 400 1579}%
\special{pa 400 1571}%
\special{fp}%
\special{pa 400 1540}%
\special{pa 400 1477}%
\special{fp}%
\special{pa 400 1447}%
\special{pa 400 1439}%
\special{fp}%
\special{pa 400 1408}%
\special{pa 400 1345}%
\special{fp}%
\special{pa 400 1315}%
\special{pa 400 1307}%
\special{fp}%
\special{pa 400 1276}%
\special{pa 400 1213}%
\special{fp}%
\special{pa 400 1183}%
\special{pa 400 1175}%
\special{fp}%
\special{pa 400 1144}%
\special{pa 400 1081}%
\special{fp}%
\special{pa 400 1051}%
\special{pa 400 1043}%
\special{fp}%
\special{pa 400 1012}%
\special{pa 400 949}%
\special{fp}%
\special{pa 400 919}%
\special{pa 400 911}%
\special{fp}%
\special{pa 400 880}%
\special{pa 400 817}%
\special{fp}%
\special{pa 400 787}%
\special{pa 400 779}%
\special{fp}%
\special{pa 400 748}%
\special{pa 400 685}%
\special{fp}%
\special{pa 400 655}%
\special{pa 400 647}%
\special{fp}%
\special{pa 400 616}%
\special{pa 400 553}%
\special{fp}%
\special{pa 400 523}%
\special{pa 400 515}%
\special{fp}%
\special{pa 400 484}%
\special{pa 400 421}%
\special{fp}%
\special{pa 400 391}%
\special{pa 400 383}%
\special{fp}%
\special{pa 400 352}%
\special{pa 400 289}%
\special{fp}%
\special{pa 400 259}%
\special{pa 400 251}%
\special{fp}%
\special{pa 400 220}%
\special{pa 400 200}%
\special{fp}%
\special{pn 8}%
\special{pa 2000 200}%
\special{pa 2000 263}%
\special{fp}%
\special{pa 2000 293}%
\special{pa 2000 302}%
\special{fp}%
\special{pa 2000 332}%
\special{pa 2000 395}%
\special{fp}%
\special{pa 2000 426}%
\special{pa 2000 434}%
\special{fp}%
\special{pa 2000 464}%
\special{pa 2000 527}%
\special{fp}%
\special{pa 2000 558}%
\special{pa 2000 566}%
\special{fp}%
\special{pa 2000 596}%
\special{pa 2000 659}%
\special{fp}%
\special{pa 2000 690}%
\special{pa 2000 698}%
\special{fp}%
\special{pa 2000 728}%
\special{pa 2000 791}%
\special{fp}%
\special{pa 2000 821}%
\special{pa 2000 829}%
\special{fp}%
\special{pa 2000 860}%
\special{pa 2000 923}%
\special{fp}%
\special{pa 2000 953}%
\special{pa 2000 961}%
\special{fp}%
\special{pa 2000 992}%
\special{pa 2000 1055}%
\special{fp}%
\special{pa 2000 1085}%
\special{pa 2000 1093}%
\special{fp}%
\special{pa 2000 1124}%
\special{pa 2000 1187}%
\special{fp}%
\special{pa 2000 1217}%
\special{pa 2000 1225}%
\special{fp}%
\special{pa 2000 1256}%
\special{pa 2000 1319}%
\special{fp}%
\special{pa 2000 1349}%
\special{pa 2000 1357}%
\special{fp}%
\special{pa 2000 1388}%
\special{pa 2000 1451}%
\special{fp}%
\special{pa 2000 1481}%
\special{pa 2000 1489}%
\special{fp}%
\special{pa 2000 1520}%
\special{pa 2000 1583}%
\special{fp}%
\special{pa 2000 1613}%
\special{pa 2000 1621}%
\special{fp}%
\special{pa 2000 1652}%
\special{pa 2000 1715}%
\special{fp}%
\special{pa 2000 1745}%
\special{pa 2000 1753}%
\special{fp}%
\special{pa 2000 1784}%
\special{pa 2000 1847}%
\special{fp}%
\special{pa 2000 1877}%
\special{pa 2000 1885}%
\special{fp}%
\special{pa 2000 1916}%
\special{pa 2000 1979}%
\special{fp}%
\special{pa 2000 2009}%
\special{pa 2000 2017}%
\special{fp}%
\special{pa 2000 2048}%
\special{pa 2000 2111}%
\special{fp}%
\special{pa 2000 2141}%
\special{pa 2000 2149}%
\special{fp}%
\special{pa 2000 2180}%
\special{pa 2000 2200}%
\special{fp}%
\put(22.4000,-12.8000){\makebox(0,0)[lb]{$r$}}%
\put(5.0000,-2.7000){\makebox(0,0)[lb]{{\large$\psi$}}}%
\put(20.3000,-11.8000){\makebox(0,0)[lb]{$1$}}%
\put(1.5000,-11.9000){\makebox(0,0)[lb]{$-1$}}%
\special{pn 20}%
\special{pa 466 200}%
\special{pa 470 296}%
\special{pa 475 366}%
\special{pa 480 416}%
\special{pa 485 455}%
\special{pa 490 486}%
\special{pa 495 513}%
\special{pa 500 536}%
\special{pa 505 556}%
\special{pa 510 574}%
\special{pa 515 590}%
\special{pa 520 605}%
\special{pa 525 619}%
\special{pa 535 643}%
\special{pa 540 654}%
\special{pa 545 664}%
\special{pa 555 682}%
\special{pa 570 706}%
\special{pa 580 720}%
\special{pa 595 738}%
\special{pa 600 743}%
\special{pa 605 749}%
\special{pa 615 759}%
\special{pa 620 763}%
\special{pa 625 768}%
\special{pa 645 784}%
\special{pa 655 790}%
\special{pa 660 794}%
\special{fp}%
\special{pn 20}%
\special{pa 678 804}%
\special{pa 680 805}%
\special{pa 685 808}%
\special{pa 690 810}%
\special{pa 695 813}%
\special{pa 702 816}%
\special{ip}%
\special{pa 721 823}%
\special{pa 725 825}%
\special{pa 730 826}%
\special{pa 735 828}%
\special{pa 740 829}%
\special{pa 745 831}%
\special{pa 746 831}%
\special{ip}%
\special{pa 766 835}%
\special{pa 775 837}%
\special{pa 780 837}%
\special{pa 790 839}%
\special{pa 792 839}%
\special{ip}%
\special{pa 813 840}%
\special{pa 835 840}%
\special{pa 839 839}%
\special{ip}%
\special{pa 860 837}%
\special{pa 860 837}%
\special{pa 865 837}%
\special{pa 885 833}%
\special{ip}%
\special{pa 906 828}%
\special{pa 910 826}%
\special{pa 915 825}%
\special{pa 930 819}%
\special{ip}%
\special{pa 949 810}%
\special{pa 950 810}%
\special{pa 955 808}%
\special{pa 960 805}%
\special{pa 965 803}%
\special{pa 972 799}%
\special{ip}%
\special{pa 990 787}%
\special{pa 995 784}%
\special{pa 1011 771}%
\special{ip}%
\special{pa 1026 758}%
\special{pa 1035 749}%
\special{pa 1040 743}%
\special{pa 1044 739}%
\special{ip}%
\special{pa 1058 723}%
\special{pa 1060 720}%
\special{pa 1070 706}%
\special{pa 1073 701}%
\special{ip}%
\special{pa 1084 684}%
\special{pa 1085 682}%
\special{pa 1095 664}%
\special{pa 1097 661}%
\special{ip}%
\special{pa 1106 642}%
\special{pa 1115 619}%
\special{pa 1115 618}%
\special{ip}%
\special{pa 1122 598}%
\special{pa 1125 590}%
\special{pa 1130 574}%
\special{pa 1130 573}%
\special{ip}%
\special{pa 1136 553}%
\special{pa 1140 536}%
\special{pa 1142 527}%
\special{ip}%
\special{pa 1146 507}%
\special{pa 1150 486}%
\special{pa 1151 481}%
\special{ip}%
\special{pa 1154 461}%
\special{pa 1155 455}%
\special{pa 1158 435}%
\special{ip}%
\special{pa 1160 414}%
\special{pa 1163 388}%
\special{ip}%
\special{pa 1165 367}%
\special{pa 1165 366}%
\special{pa 1167 341}%
\special{ip}%
\special{pa 1168 320}%
\special{pa 1170 296}%
\special{pa 1170 294}%
\special{ip}%
\special{pa 1171 273}%
\special{pa 1172 247}%
\special{ip}%
\special{pa 1173 226}%
\special{pa 1174 200}%
\special{ip}%
\special{pn 20}%
\special{pa 1907 220}%
\special{pa 1908 246}%
\special{ip}%
\special{pa 1909 266}%
\special{pa 1910 292}%
\special{ip}%
\special{pa 1911 312}%
\special{pa 1913 337}%
\special{ip}%
\special{pa 1914 358}%
\special{pa 1915 366}%
\special{pa 1917 383}%
\special{ip}%
\special{pa 1919 403}%
\special{pa 1920 416}%
\special{pa 1922 428}%
\special{ip}%
\special{pa 1924 449}%
\special{pa 1925 455}%
\special{pa 1928 474}%
\special{ip}%
\special{pa 1931 494}%
\special{pa 1935 513}%
\special{pa 1936 519}%
\special{ip}%
\special{pa 1941 539}%
\special{pa 1945 556}%
\special{pa 1947 563}%
\special{ip}%
\special{pa 1953 583}%
\special{pa 1955 590}%
\special{pa 1960 605}%
\special{pa 1961 607}%
\special{ip}%
\special{pa 1968 626}%
\special{pa 1975 643}%
\special{pa 1978 649}%
\special{ip}%
\special{pa 1987 668}%
\special{pa 1995 682}%
\special{pa 2000 690}%
\special{ip}%
\special{pa 2011 707}%
\special{pa 2020 720}%
\special{pa 2026 727}%
\special{ip}%
\special{pa 2039 742}%
\special{pa 2040 743}%
\special{pa 2045 749}%
\special{pa 2055 759}%
\special{pa 2057 761}%
\special{ip}%
\special{pa 2072 774}%
\special{pa 2085 784}%
\special{pa 2093 789}%
\special{ip}%
\special{pa 2110 800}%
\special{pa 2115 803}%
\special{pa 2120 805}%
\special{pa 2125 808}%
\special{pa 2130 810}%
\special{pa 2133 812}%
\special{ip}%
\special{pa 2151 820}%
\special{pa 2165 825}%
\special{pa 2170 826}%
\special{pa 2175 828}%
\special{pa 2175 828}%
\special{ip}%
\special{pa 2195 833}%
\special{pa 2215 837}%
\special{pa 2220 837}%
\special{ip}%
%
\special{pn 4}%
\special{pa 440 200}%
\special{pa 440 1200}%
\special{da 0.070}%
\put(6.0000,-13.1000){\makebox(0,0)[lb]{{\tiny $r_0$}}}%
\put(12.2000,-8.5000){\makebox(0,0)[lb]{{\tiny $\psi(r_0)$}}}%
\put(4.2000,-13.1000){\makebox(0,0)[lb]{{\tiny $r_1$}}}%
%
\special{pn 4}%
\special{pa 650 790}%
\special{pa 650 1200}%
\special{dt 0.045}%
\special{pa 650 790}%
\special{pa 1200 790}%
\special{dt 0.045}%
%
\special{pn 8}%
\special{pa 200 1200}%
\special{pa 2200 1200}%
\special{fp}%
\special{sh 1}%
\special{pa 2200 1200}%
\special{pa 2133 1180}%
\special{pa 2147 1200}%
\special{pa 2133 1220}%
\special{pa 2200 1200}%
\special{fp}%
%
\special{pn 8}%
\special{pa 1200 2200}%
\special{pa 1200 200}%
\special{fp}%
\put(10.5000,-1.8000){\makebox(0,0)[lb]{$r=R$}}%
\end{picture}}}
\caption{The behavior of the graph of $\psi$ in Lemma \ref{thm:psi-shape3}}
\end{figure}
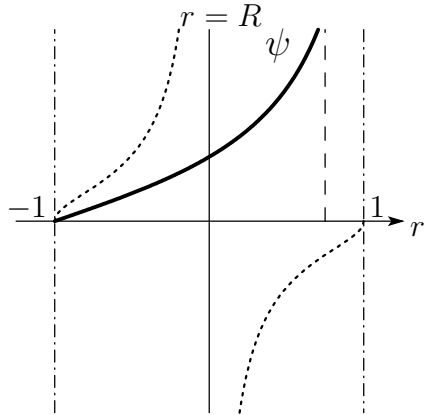

Since the existence of the solution $\psi$ of (\ref{eq:isopara-soliton-sphere3}) which is definded to $r=-1$ could not be excluded, we consider that case.

\begin{lemma}\label{thm:psi-shape4}
If $\psi$ is defined to $r=-1$, $\psi(-1)=0$ and $V'(-1)=\frac{1}{k(k+(n-1)(1+R))}$
\end{lemma}
\begin{proof}
It is clear that $\psi(-1)=0$.
By $V'(r)=\frac{1}{k\sqrt{(1-r^2)}}\psi(r)$, we find
\begin{align*}
\frac{\psi'(r)}{(k\sqrt{1-r^2}~)'}&=-\frac{1}{k^2r}\left(k^2(1-r^2)V'(r)^2+1\right)\left(k(n-1)(r-R)V'(r)+1\right)\\
&\rightarrow-\frac{1}{k^2}\left(k(n-1)(1+R)V'(-1)-1\right)\quad(~r~\downarrow~-1~)
\end{align*}
By the l'H${\rm\hat{o}}$pital's rule, we find $V'(-1)=\frac{1}{k(k+(n-1)(1+R))}$.
\end{proof}

\begin{figure}[H]
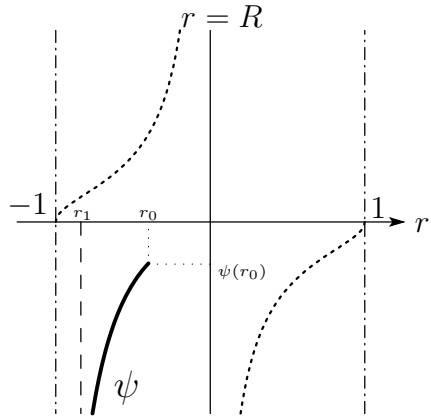

\centering
\scalebox{1.0}{{\unitlength 0.1in%
}}
\caption{The behavior of the graph of $\psi$ in Lemma \ref{thm:psi-shape4}}
\end{figure}

Also, in the case where $\psi<0$, by proofs similar to Lemmas \ref{thm:psi-shape1}-\ref{thm:psi-shape4}, we obtain following lemmas.
\begin{lemma}\label{thm:psi-shape5}
If there exists $r_0\in(-1,R)$ with $\psi(r_0)<0$, there exists $r_1\in(-1,r_0)$ such that 
\begin{equation*}
\lim_{r\downarrow r_1}\psi(r)=-\infty.
\end{equation*}
\end{lemma}

\begin{figure}[H]
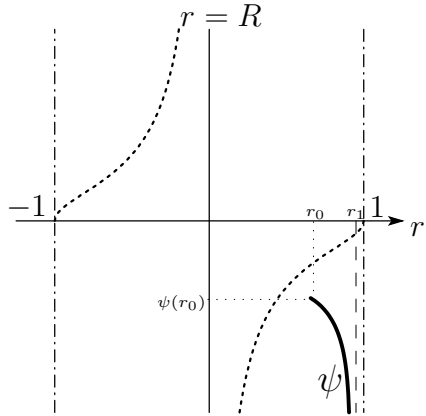

\centering
\scalebox{1.0}{{\unitlength 0.1in%
%
}}
\caption{The behavior of the graph of $\psi$ in Lemma \ref{thm:psi-shape5}}
\end{figure}

\begin{lemma}\label{thm:psi-shape6}
If there exists $r_0\in\left(R,1\right)$ with $0>\psi(r_0)>\eta(r_0)$, there exists $C\in\left(-\infty,\psi(r_0)\right)$ such that
\begin{equation*}
\lim_{r\downarrow R}\psi(r)=C.
\end{equation*}
\end{lemma}

\begin{figure}[H]
\centering
\scalebox{1.0}{{\unitlength 0.1in%
%
}}
\caption{The behavior of the graph of $\psi$ in Lemma \ref{thm:psi-shape6}}
\end{figure}

\begin{lemma}\label{thm:psi-shape7}
If there exists $r_0\in(R,1)$ with $\psi(r_0)<\eta(r_0)$, there exists $r_1\in(r_0,1)$ such that 
\begin{equation*}
\lim_{r\uparrow r_1}\psi(r)=-\infty.
\end{equation*}
\end{lemma}

\begin{figure}[H]
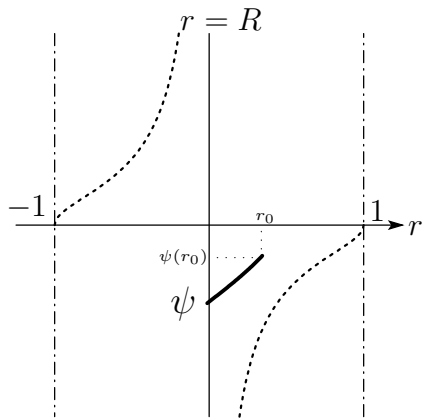

\centering
\scalebox{1.0}{{\unitlength 0.1in%
%
}}
\caption{The behavior of the graph of $\psi$ in Lemma \ref{thm:psi-shape7}}
\end{figure}

\begin{lemma}\label{thm:psi-shape8}
If $\psi$ is defined to $r=1$, $\psi(1)=0$ and $V'(1)=-\frac{1}{k(k+(n-1)(1-R))}$
\end{lemma}

\begin{figure}[H]
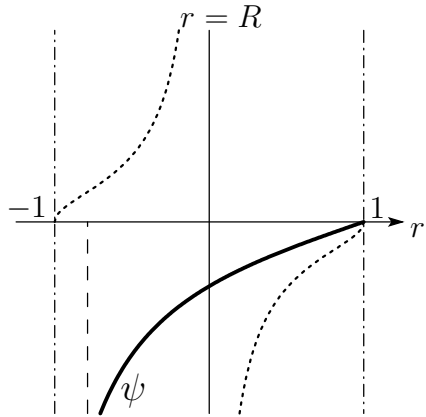

\centering
\scalebox{1.0}{{\unitlength 0.1in%
%
}}
\caption{The behavior of the graph of $\psi$ in Lemma \ref{thm:psi-shape8}}
\end{figure}

By Lemmas \ref{thm:sign}-\ref{thm:psi-shape8}, we obtain the following proposition for the behavior of the graph of $\psi$.

\begin{proposition}\label{thm:graph-psi}
For the solution $\psi$ of the equation {\rm (\ref{eq:isopara-soliton-sphere3})}, the behavior of the graph of $\psi$ is like one of {\rm Figures \ref{psiex1}}-{\rm\ref{psiex4-2}}.
\end{proposition}

\begin{figure}[H]
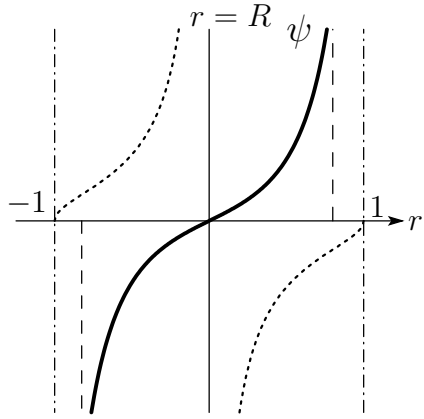

\centering
\scalebox{1.0}{{\unitlength 0.1in%
%
}}
\caption{The graph of $\psi$ (Type I$''$)}
\label{psiex1}
\end{figure}

\begin{figure}[H]
\centering
\begin{minipage}{0.49\columnwidth}
\centering
\scalebox{1.0}{{\unitlength 0.1in%
%
}}
\caption{The graph of $\psi$ (Type II$''$)}
\label{psiex2}
\end{minipage}
\begin{minipage}{0.49\columnwidth}
\centering
\scalebox{1.0}{{\unitlength 0.1in%
%
}}
\caption{The graph of $\psi$ (Type III$''$)}
\label{psiex2-2}
\end{minipage}
\end{figure}

\begin{figure}[H]
\centering
\begin{minipage}{0.49\columnwidth}
\centering
\scalebox{1.0}{{\unitlength 0.1in%
%
}}
\caption{The graph of $\psi$ (Type IV$''$)}
\label{psiex3}
\end{minipage}
\begin{minipage}{0.49\columnwidth}
\centering
\scalebox{1.0}{{\unitlength 0.1in%
%
}}
\caption{The graph of $\psi$ (Type V$''$)}
\label{psiex3-2}
\end{minipage}
\end{figure}

\begin{figure}[H]
\centering
\begin{minipage}{0.49\columnwidth}
\centering
\scalebox{1.0}{{\unitlength 0.1in%
%
}}
\caption{The graph of $\psi$ (Type VI$''$)}
\label{psiex4}
\end{minipage}
\begin{minipage}{0.49\columnwidth}
\centering
\scalebox{1.0}{{\unitlength 0.1in%
%
}}
\caption{The graph of $\psi$ (Type VII$''$)}
\label{psiex4-2}
\end{minipage}
\end{figure}

For the graph of $\phi$ in Proposition \ref{thm:graph-psi}, we have not yet show whether $\psi$ in the case of {\rm Figures \ref{psiex1}} and {\rm \ref{psiex4-2}} exists or not.
By the following lemma, we obtain the existence.

\begin{lemma}
The solutions $\psi$ of the equation {\rm (\ref{eq:isopara-soliton-sphere3})} in {\rm Figures \ref{psiex4}} and {\rm \ref{psiex4-2}} exist.
\end{lemma}
\begin{proof}
For the set $S$ of all solutions of the equation {\rm (\ref{eq:isopara-soliton-sphere3})}, we define sets $S_1$, $S_2$, $S_3\subset S$  by
\begin{align*}
&S_1:=\{\psi\in S\vert\exists r_0\in(-1,1):\psi(r_0)=0\}\\
&S_2:=\{\psi\in S\vert\exists r_0\in(-1,1):\psi(r_0)=\eta(r_0)\}\\
&S_3:=\{\psi\in S\vert\psi(1)=0~or~\psi(-1)=0\}.
\end{align*}
Then, we have
\begin{equation*}
(-1,1)\times\mathbb{R}=\cup_{\psi\in S_1\cup S_2\cup S_3}{\rm Im}(\psi).
\end{equation*}
Since $\cup_{\psi\in S_1}{\rm Im}(\psi)$ and  $\cup_{\psi\in S_2}{\rm Im}(\psi)$ are open sets and $(-1,1)\times\mathbb{R}$ is connected, we find $S_3$ is not empty set.
Therefore, we obtain the statement of this lemma.
\end{proof}

Define $\zeta(r)=-\frac{1}{k(n-1)(r-R)}$.
By $V'(r)=\frac{1}{k\sqrt{1-r^2}}\psi(r)$ and Proposition \ref{thm:graph-psi}, we have the following proposition for the behavior of the graph of $V'$.
Besides, by Proposition \ref{thm:graph-derivative}, we obtain Theorem \ref{thm:shape-graph}.
\begin{proposition}\label{thm:graph-derivative}
For the solution $V$ of the equation {\rm (\ref{eq:isopara-soliton-sphere2})}, the behavior of the graph of $V'$ is like one of {\rm Figures \ref{V'ex1}}-{\rm \ref{V'ex4-2}}.
Here, the dotted curve in {\rm Figures \ref{V'ex1}}-{\rm \ref{V'ex4-2}} is the graph of $\zeta$.
\end{proposition}

\begin{figure}[H]
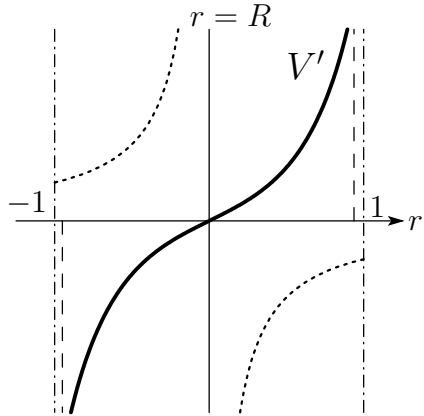

\centering
\scalebox{1.0}{{\unitlength 0.1in%
}}
\caption{The graph of $V'$ (Type I$'$)}
\label{V'ex1}
\end{figure}

\begin{figure}[H]
\centering
\begin{minipage}{0.49\columnwidth}
\centering
\scalebox{1.0}{{\unitlength 0.1in%
%
}}
\caption{The graph of $V'$ (Type II$'$)}
\label{V'ex2}
\end{minipage}
\begin{minipage}{0.49\columnwidth}
\centering
\scalebox{1.0}{{\unitlength 0.1in%
%
}}
\caption{The graph of $V'$ (Type III$'$)}
\label{V'ex2-2}
\end{minipage}
\end{figure}

\begin{figure}[H]
\centering
\begin{minipage}{0.49\columnwidth}
\centering
\scalebox{1.0}{{\unitlength 0.1in%
%
}}
\caption{The graph of $V'$ (Type IV$'$)}
\label{V'ex3}
\end{minipage}
\begin{minipage}{0.49\columnwidth}
\centering
\scalebox{1.0}{{\unitlength 0.1in%
%
}}
\caption{The graph of $V'$ (Type V$'$)}
\label{V'ex3-2}
\end{minipage}
\end{figure}

\begin{figure}[H]
\centering
\begin{minipage}{0.49\columnwidth}
\centering
\scalebox{1.0}{{\unitlength 0.1in%
%
}}
\caption{The graph of $V'$ (Type VI$'$)}
\label{V'ex4}
\end{minipage}
\begin{minipage}{0.49\columnwidth}
\centering
\scalebox{1.0}{{\unitlength 0.1in%
%
}}
\caption{The graph of $V'$ (Type VII$'$)}
\label{V'ex4-2}
\end{minipage}
\end{figure}

\section{The domain of the function $u$ in Theorem \ref{thm:shape-graph}}

In this section, we investigate the domain of the function $u=V\circ r$ over $M\subset\mathbb{S}^n$ in Theorem \ref{thm:shape-graph} in the case where the number $k$ of distinct principal curvatures of the isoparametric hypersurface for $r$ is $1$, $2$ or $3$.
From the result of Theorem \ref{thm:shape-graph}, we find that $M$ does not contain some tubular neighborhoods of the focal submanifolds $r^{-1}(1)$ and $r^{-1}(-1)$ in case that the type of $V$ in Theorem \ref{thm:shape-graph} is I-V.
Also, we find that $M$ contains $r^{-1}(-1)$ and does not contain a tubular neighborhood of the focal submanifold $r^{-1}(1)$ in case that the type of $V$ is VI and $M$ contains $r^{-1}(1)$ and does not contain a tubular neighborhood of the focal submanifold $r^{-1}(-1)$ in case that the type of $V$ is VII.

When $k=1$, the isoparametric function $r$ is defined by
$$
r(x_1,\cdots,x_{n+1})=x_{n+1}~~\quad(x_1,\cdots,x_{n+1})\in\mathbb{S}^n.
$$
Therefore, from the result of Theorem \ref{thm:shape-graph}, the domain $M$ of $u$ is an open set of $\mathbb{S}^n$ including the set $\{(x_1,\cdots,x_n,0)\in\mathbb{R}^{n+1}\vert~x_1^2+\cdots+x_n^2=1\}\subset\mathbb{S}^n$.
Also, as $p=(0,\cdots,0,1), q=(0,\cdots,0,-1)$, we find that $p,q\notin M$ in case that the type of $V$ in Theorem \ref{thm:shape-graph} is I-V, $p\notin M$, $q\in M$ in case that the type of $V$ is VI and $p\in M$, $q\notin M$ in case that the type of $V$ is VII.

When $k=2$, the isoparametric function $r$ is defined by
$$
r(x_1,\cdots,x_{n+1})=\sum_{i=1}^lx_i^2-\sum_{i=l+1}^{n+1}x_i^2~~\quad(x_1,\cdots,x_{n+1})\in\mathbb{S}^n.
$$
Here, $l\in\{1,\cdots,n\}$.
Since $r^{-1}(t)=\{(x,y)\in\mathbb{R}^l\times\mathbb{R}^{n-l+1}\vert~\|x\|^2=\frac{1+t}{2},\|y\|^2=\frac{1-t}{2}\}$ for $t\in(-1,1)$, as $S_\theta:=\{((\cos{\theta},0,\cdots,0)\mathbf{A},(\sin{\theta},0,\cdots,0)\mathbf{B})\in\mathbb{R}^l\times\mathbb{R}^{n-l+1}~\vert~\mathbf{A}\in SO(l-1),\mathbf{B}\in SO(n-l)\}$, we obtain that $r^{-1}(t)=S_{\theta_t}$ for $\theta_t\in(0,\frac{\pi}{2})$ with $\cos{\theta_t}=\sqrt{\frac{1+t}{2}}$ and $\sin{\theta_t}=\sqrt{\frac{1-t}{2}}$.
Therefore, from the result of Theorem \ref{thm:shape-graph}, we find that the domain $M$ is the open set of $\mathbb{S}^n$ including $S_{\theta_R}$.
Also, we find that $M=\cup_{\theta\in I}S_\theta$ for an interval $I\subset(0,\frac{\pi}{2})$ in case that the type of $V$ in Theorem \ref{thm:shape-graph} is I-V, $M=\cup_{\theta\in(a,\frac{\pi}{2}]}S_\theta$ for some $a\in(0,\frac{\pi}{2})$ in case that the type of $V$ is VI and $M=\cup_{\theta\in[0,a)}S_\theta$ for $a\in(0,\frac{\pi}{2})$ in case that the type of $V$ is VII.

When $k=3$, an isoparametric hypersurface is a principal orbit of the isotropy representation of the rank two symmetric space $G/K=SU(3)/SO(3)$, $(SU(3)\times SU(3))/SU(3)$, $SU(6)/Sp(3)$ or $E_6/F_4$.
Since the principal orbit of the isotropy representation intersects with the Weyl domain $C$ at only one point, we find that there exists an open subset $U\subset\overline{C}\cap\mathbb{S}^n$ such that $K\cdot U$ is equal to $M$.
Here, $T_e(G/K)$ for $e\in G/K$ is identified with $\mathbb{R}^{n+1}$.
Also, we find that $M\subset K\cdot C$ in case that the type of $V$ in Theorem \ref{thm:shape-graph} is I-V and $M\cap(\overline{C}\setminus C)\neq\emptyset$  in case that the type of $V$ is $VI$ or $VII$.

In the rest of this paper, we shall give explicit descriptions of Weyl domains for the symmetric space $G/K=SU(3)/SO(3)$, $(SU(3)\times SU(3))/SU(3)$ or $SU(6)/Sp(3)$.
Define $\mathfrak{g}$, $\mathfrak{k}$ and $\mathfrak{p}$ by $\mathfrak{g}:=Lie G$, $\mathfrak{k}:=Lie K$ and $\mathfrak{g}=\mathfrak{k}\oplus\mathfrak{p}$.
Denote by $\mathfrak{a}$ the maximal abelian subspace of $\mathfrak{p}$.
When $G/K=SU(3)/SO(3)$, we have that $\mathfrak{p}=\{\mathbf{A}~:~3\times3$ symmetric space purely imaginary matrix such that the trace of $\mathbf{A} =0\}$ and the diagonal matrices in $\mathfrak{p}$ form $\mathfrak{a}$.
So, we obtain
\begin{equation*}
\mathfrak{a}=\left\{\left.
\begin{pmatrix}
\sqrt{-1}a & 0 & 0\\
0 &\sqrt{-1}b & 0\\
0 & 0 & -\sqrt{-1}(a+b)
\end{pmatrix}
\right\vert a,b\in\mathbb{R}
\right\}.
\end{equation*}
Define $e_i~(1\le i\le3)$ as $e_i(\mathbf{A})$ is the diagonal element of $\mathbf{A}$.
Then, we find that for the basis $\{e_1,e_2\}$ the positive restricted root system $\triangle_+=\{\sqrt{-1}(e_1-e_2), \sqrt{-1}(e_1-e_3),\sqrt{-1}(e_2-e_3)\}$.
Since the Killing form $B$ is defined by $B(\mathbf{X},\mathbf{Y})=6Tr(\mathbf{X}\mathbf{Y})$, as
\begin{align*}
\mathbf{A_1}=
\begin{pmatrix}
\sqrt{-1} & 0 & 0\\
0 &-\sqrt{-1} & 0\\
0 & 0 & 0
\end{pmatrix},~
\mathbf{A_2}=
\begin{pmatrix}
\sqrt{-1} & 0 & 0\\
0 & 0 & 0\\
0 & 0 & -\sqrt{-1}
\end{pmatrix},~
\mathbf{A_3}=
\begin{pmatrix}
0 & 0 & 0\\
0 & \sqrt{-1} & 0\\
0 & 0 & -\sqrt{-1}
\end{pmatrix},
\end{align*}
we obtain the Weyl domain $C$ in Figure \ref{Weyl}.
Here, for the angle $\theta_{ij}$ with respect to $\mathbf{A_i}$ and $\mathbf{A_j}$, we find $\theta_{12}=\theta_{23}=\frac{\pi}{3}$ and $\theta_{13}=\frac{2\pi}{3}$.

\begin{figure}[H]
\centering
\scalebox{1.0}{{\unitlength 0.1in%
\begin{picture}(23.3000,20.0000)(2.0000,-22.0000)%
%
\special{pn 8}%
\special{pa 200 200}%
\special{pa 2200 200}%
\special{pa 2200 2200}%
\special{pa 200 2200}%
\special{pa 200 200}%
\special{pa 2200 200}%
\special{fp}%
%
\special{pn 13}%
\special{pa 1200 1200}%
\special{pa 1800 1200}%
\special{fp}%
\special{sh 1}%
\special{pa 1800 1200}%
\special{pa 1733 1180}%
\special{pa 1747 1200}%
\special{pa 1733 1220}%
\special{pa 1800 1200}%
\special{fp}%
%
\special{pn 13}%
\special{pa 1200 1200}%
\special{pa 1600 700}%
\special{fp}%
\special{sh 1}%
\special{pa 1600 700}%
\special{pa 1543 740}%
\special{pa 1567 742}%
\special{pa 1574 765}%
\special{pa 1600 700}%
\special{fp}%
%
\special{pn 13}%
\special{pa 1200 1200}%
\special{pa 800 700}%
\special{fp}%
\special{sh 1}%
\special{pa 800 700}%
\special{pa 826 765}%
\special{pa 833 742}%
\special{pa 857 740}%
\special{pa 800 700}%
\special{fp}%
%
\special{pn 8}%
\special{pa 1200 1200}%
\special{pa 2200 720}%
\special{dt 0.045}%
%
\special{pn 8}%
\special{pa 1200 1200}%
\special{pa 1200 200}%
\special{dt 0.045}%
%
\special{pn 4}%
\special{pa 1660 200}%
\special{pa 1200 660}%
\special{fp}%
\special{pa 1720 200}%
\special{pa 1200 720}%
\special{fp}%
\special{pa 1780 200}%
\special{pa 1200 780}%
\special{fp}%
\special{pa 1840 200}%
\special{pa 1200 840}%
\special{fp}%
\special{pa 1900 200}%
\special{pa 1200 900}%
\special{fp}%
\special{pa 1960 200}%
\special{pa 1200 960}%
\special{fp}%
\special{pa 2020 200}%
\special{pa 1200 1020}%
\special{fp}%
\special{pa 1540 740}%
\special{pa 1200 1080}%
\special{fp}%
\special{pa 1380 960}%
\special{pa 1200 1140}%
\special{fp}%
\special{pa 2080 200}%
\special{pa 1580 700}%
\special{fp}%
\special{pa 2140 200}%
\special{pa 1590 750}%
\special{fp}%
\special{pa 2190 210}%
\special{pa 1260 1140}%
\special{fp}%
\special{pa 2200 260}%
\special{pa 1320 1140}%
\special{fp}%
\special{pa 2200 320}%
\special{pa 1440 1080}%
\special{fp}%
\special{pa 2200 380}%
\special{pa 1550 1030}%
\special{fp}%
\special{pa 2200 440}%
\special{pa 1670 970}%
\special{fp}%
\special{pa 2200 500}%
\special{pa 1780 920}%
\special{fp}%
\special{pa 2200 560}%
\special{pa 1900 860}%
\special{fp}%
\special{pa 2200 620}%
\special{pa 2010 810}%
\special{fp}%
\special{pa 2200 680}%
\special{pa 2130 750}%
\special{fp}%
\special{pa 1570 770}%
\special{pa 1500 840}%
\special{fp}%
\special{pa 1600 200}%
\special{pa 1200 600}%
\special{fp}%
\special{pa 1540 200}%
\special{pa 1200 540}%
\special{fp}%
\special{pa 1480 200}%
\special{pa 1200 480}%
\special{fp}%
\special{pa 1420 200}%
\special{pa 1200 420}%
\special{fp}%
\special{pa 1360 200}%
\special{pa 1200 360}%
\special{fp}%
\special{pa 1300 200}%
\special{pa 1200 300}%
\special{fp}%
\special{pa 1240 200}%
\special{pa 1200 240}%
\special{fp}%
\put(25.3000,-3.9000){\makebox(0,0)[lb]{{\Large $C$}}}%
\put(18.4000,-12.5000){\makebox(0,0)[lb]{$\mathbf{A_1}$}}%
\put(16.3000,-6.7000){\makebox(0,0)[lb]{$\mathbf{A_2}$}}%
\put(6.6000,-6.8000){\makebox(0,0)[lb]{$\mathbf{A_3}$}}%
%
\special{pn 8}%
\special{pa 2500 340}%
\special{pa 2070 340}%
\special{fp}%
\special{sh 1}%
\special{pa 2070 340}%
\special{pa 2137 360}%
\special{pa 2123 340}%
\special{pa 2137 320}%
\special{pa 2070 340}%
\special{fp}%
\put(3.1000,-21.0000){\makebox(0,0)[lb]{{\LARGE $\mathfrak{a}$}}}%
\end{picture}}}
\caption{The Weyl domain $C$}
\label{Weyl}
\end{figure}
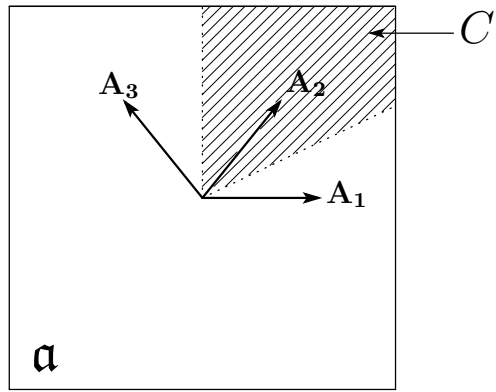

When $G/K=(SU(3)\times SU(3))/SU(3)$, we have that 
\begin{equation*}
\mathfrak{p}=\left\{\left.\begin{pmatrix}
\mathbf{A} & \mathbf{O} \\
\mathbf{O} & -\mathbf{A} \\
\end{pmatrix}
\right\vert
\mathbf{A}:3\times3~skew~Hermitian~matrix,~the~trace~of~\mathbf{A}=0\right\}
\end{equation*}
and the diagonal matrices in $\mathfrak{p}$ form $\mathfrak{a}$.

Also, When $G/K=SU(6)/Sp(3)$, we have that 
\begin{align*}
\mathfrak{p}=\left\{\left.\begin{pmatrix}
\mathbf{A} & \mathbf{B} \\
\mathbf{\overline{B}} & -\mathbf{\overline{A}} \\
\end{pmatrix}
\right\vert
\begin{matrix}
\mathbf{A}:3\times3~skew~Hermitian~matrix,~the~trace~of~\mathbf{A}=0,\\
~\mathbf{B}~:~3\times3~skew~symmetric~matrix
\end{matrix}\right\}
\end{align*}
and the diagonal matrices in $\mathfrak{p}$ form $\mathfrak{a}$.

In a similar way we obtain in the case that $G/K=SU(3)/SO(3)$, that the Weyl domains $C$ for $(SU(3)\times SU(3))/SU(3)$ and $SU(6)/Sp(3)$ are as in Figure \ref{Weyl}.

\section*{Acknowledgement}

I would like to thank my supervisor Naoyuki Koike for helpful support and valuable comment.

\end{document}